\documentclass[secthm,seceqn,amsthm,ussrhead,12pt]{amsart}
\usepackage[utf8]{inputenc}
\usepackage[english]{babel}
\usepackage{amssymb,amsmath,amsthm,amsfonts,xcolor,enumerate,hyperref,comment,longtable,cleveref}

\usepackage{times}
\usepackage{cite}
\usepackage{pdflscape}
\usepackage{ulem}
\usepackage[mathcal]{euscript}
\usepackage{tikz}
\usepackage{hyperref}
\usepackage{cancel}
\usepackage{stmaryrd}

\usetikzlibrary{arrows}

\newtheorem{theorem}{Theorem}
\newtheorem{lemma}[theorem]{Lemma}
\newtheorem{proposition}[theorem]{Proposition}

\newtheorem{definition}[theorem]{Definition}
\newtheorem{corollary}[theorem]{Corollary}

\newcommand{\orb}{\mathrm{Orb}}

\newcommand{\af}{\alpha}
\newcommand{\bt}{\beta}

\usepackage{stmaryrd}
\usepackage{xcolor}

\begin{document}
\noindent{\Large
The algebraic classification of nilpotent commutative $\mathfrak{CD}$-algebras}

   \medskip
\ 

\ 
   \medskip

   {\bf  Doston Jumaniyozov,
   Ivan   Kaygorodov \&
   Abror Khudoyberdiyev}

\

\

\

\noindent{\bf Abstract}:
{\it 
An algebraic classification of complex $5$-dimensional nilpotent commutative $\mathfrak{CD}$-algebras is given.
This classification is based on an algebraic classification of complex $5$-dimensional  nilpotent Jordan algebras.}

\

\ 

\


\medskip


\medskip

\medskip

\medskip

\medskip

\section*{Introduction}
There are many results related to the algebraic and geometric 
classification
of low-dimensional algebras in the varieties of Jordan, Lie, Leibniz and 
Zinbiel algebras;
for algebraic classifications  see, for example, 
\cite{ack, cfk19, ckkk20, kks19, degr3, ck13, usefi1, degr2, degr1, hac18, ikm19,  kkk18, kpv19, kv16}; 
for geometric classifications and descriptions of degenerations see, for example, 
\cite{pilar}.
In the present paper, we give an algebraic classification of nilpotent commutative $\mathfrak{CD}$-algebras.
 This is a new class of non-associative algebras introduced in \cite{ack}.
The idea of the definition of a $\mathfrak{CD}$-algebra comes from the following property of Jordan and Lie algebras: {\it the commutator of any pair of multiplication operators is a derivation}.
 This gives three identities of degree four,  which reduce to only one identity of degree four in the commutative or anticommutative case.
Commutative and anticommutative  $\mathfrak{CD}$-algebras are related to many interesting varieties of algebras.
 Thus, anticommutative  $\mathfrak{CD}$-algebras is a generalization of Lie algebras, 
 containing the intersection of Malcev and Sagle algebras as a proper subvariety.   Moreover, the following intersections of varieties coincide:
Malcev and Sagle algebras; 
Malcev and anticommutative  $\mathfrak{CD}$-algebras; 
Sagle and anticommutative  $\mathfrak{CD}$-algebras.
On the other hand, 
the variety of anticommutative  $\mathfrak{CD}$-algebras is a proper subvariety   of 
the varieties of binary Lie algebras 
and almost Lie algebras \cite{kz}.
The variety of anticommutative  $\mathfrak{CD}$-algebras coincides with the intersection of the varieties of binary Lie algebras and almost Lie algebras.
Commutative  $\mathfrak{CD}$-algebras is a generalization of Jordan algebras, 
which is a generalization of associative commutative algebras.
On the other hand, the variety of commutative  $\mathfrak{CD}$-algebras is also known as the variety of almost-Jordan algebras, which states in the bigger variety of generalized almost-Jordan algebras \cite{arenas,hl,fl15,labra}.
 The $n$-ary  version of commutative  $\mathfrak{CD}$-algebras was introduced in a recent paper by 
Kaygorodov, Pozhidaev and Saraiva \cite{kps19}. 

\newpage
The variety of almost-Jordan algebras is the variety of commutative algebras, 
satisfying \[2((yx)x)x+yx^3=3(yx^2)x.\] 
This present identity appeared in a paper of Osborn \cite{os65},
during the study of identities of degree less than or equal to $4$ of non-associative algebras. The identity is a linearized form of the Jordan identity.
The systematic study of almost-Jordan algebras was initiated in the next paper of Osborn \cite{osborn65} and it was continued in some papers of Petersson \cite{petersson, petersson67}, Osborn \cite{osborn69}, and Sidorov \cite{Sidorov_1981}
(sometimes, it was  called as Lie triple algebras).
Hentzel and  Peresi proved that every semiprime almost-Jordan ring is Jordan \cite{peresi}.
After that, 
Labra and Correa
proved  that a finite-dimensional almost-Jordan right-nilalgebra is nilpotent \cite{cl09,cl09-2}.
Assosymmetric algebras under the symmetric product give almost-Jordan algebras  \cite{askar18}.

The algebraic classification of nilpotent algebras will be achieved by the calculation of central extensions of algebras from the same variety which have a smaller dimension.
Central extensions of algebras from various varieties were studied, for example, in \cite{ss78,ck13,klp20,zusmanovich,kkl18,omirov}.
Skjelbred and Sund \cite{ss78} used central extensions of Lie algebras to classify nilpotent Lie algebras.
Using the same method,  
all non-Lie central extensions of  all $4$-dimensional Malcev algebras \cite{hac16},
all non-associative central extensions of all $3$-dimensional Jordan algebras \cite{ha17},
all anticommutative central extensions of $3$-dimensional anticommutative algebras \cite{cfk182},
all central extensions of $2$-dimensional algebras \cite{cfk18}
and some others were described.
One can also look at the classification of
$3$-dimensional nilpotent algebras \cite{fkkv},
$4$-dimensional nilpotent associative algebras \cite{degr1},
$4$-dimensional nilpotent Novikov algebras \cite{kkk18},
$4$-dimensional nilpotent bicommutative algebras \cite{kpv19},
$4$-dimensional nilpotent commutative algebras in \cite{fkkv},
$4$-dimensional nilpotent $\mathfrak{CD}$-algebras in \cite{kk20},
$5$-dimensional nilpotent restricted Lie agebras \cite{usefi1},
$5$-dimensional nilpotent Jordan algebras \cite{ha16},
$5$-dimensional nilpotent anticommutative algebras \cite{fkkv},
$6$-dimensional nilpotent Lie algebras \cite{degr3, degr2},
$6$-dimensional nilpotent Malcev algebras \cite{hac18},
$6$-dimensional nilpotent Tortkara algebras \cite{gkk},
$6$-dimensional nilpotent binary Lie algebras \cite{ack},
$6$-dimensional nilpotent anticommutative $\mathfrak{CD}$-algebras \cite{ack},
$8$-dimensional dual mock-Lie algebras \cite{lisa}.

\paragraph{\bf Motivation and contextualization} 
Given algebras ${\bf A}$ and ${\bf B}$ in the same variety, we write ${\bf A}\to {\bf B}$ and say that ${\bf A}$ {\it degenerates} to ${\bf B}$, or that ${\bf A}$ is a {\it deformation} of ${\bf B}$, if ${\bf B}$ is in the Zariski closure of the orbit of ${\bf A}$ (under the base-change action of the general linear group). The study of degenerations of algebras is very rich and closely related to deformation theory, in the sense of Gerstenhaber. It offers an insightful geometric perspective on the subject and has been the object of a lot of research.
In particular, there are many results concerning degenerations of algebras of small dimensions in a  variety defined by a set of identities.
One of the main problems of the {\it geometric classification} of a variety of algebras is a description of its irreducible components. In the case of finitely-many orbits (i.e., isomorphism classes), the irreducible components are determined by the rigid algebras --- algebras whose orbit closure is an irreducible component of the variety under consideration. 
The algebraic classification of complex  $5$-dimensional nilpotent commutative $\mathfrak{CD}$-algebras gives  a way to obtain a geometric classification of  all
complex $5$-dimensional nilpotent commutative $\mathfrak{CD}$-algebras, 
as well as an algebraic classification of all complex $5$-dimensional nilpotent commutative algebras.

\ 

\paragraph{\bf Main result}
The main result of the paper is the complete classification of complex $5$-dimensional nilpotent commutative $\mathfrak{CD}$-algebras.
The full list of non-isomorphic algebras consists two parts:
\begin{enumerate}
    \item Jordan algebras were classified in \cite{ha16};
    
    \item Non-Jordan commutative $\mathfrak{CD}$-algebras were found in the last part of the present paper [Theorem \ref{teor}].
    
 \end{enumerate}

 \newpage
\section{Preliminaries}
The class of  $\mathfrak{CD}$-algebras is defined by the 
property that the commutator of any pair of multiplication operators is a derivation; 
namely, an algebra $\mathfrak{A}$ is a $\mathfrak{CD}$-algebra if and only if  
\[ [T_x,T_y]   \in \mathfrak{Der} (\mathfrak{A}),\]
for all $x,y \in \mathfrak{A}$, where  $T_z \in \{ R_z,L_z\}$. Here we use the notation $R_z$ (resp. $L_z$) for the operator of right (resp. left) multiplication in $\mathfrak{A}$. We will denote the variety of commutative $\mathfrak{CD}$-algebras by $\mathfrak{CCD}$.
In terms of identities, the class of $\mathfrak{CCD}$-algebras is defined by the following identity:
\[((xy)a)b  + ((xb)a)y + x((yb)a)  = ((xy)b)a +  ((xa)b)y+ x((ya)b).\]

Let us define $ G_{x}: \mathfrak{C}\times \mathfrak{C}\times \mathfrak{C}\rightarrow \mathfrak{C} $ by the following way:
\[G_{x}(y,z,t)=(yz,x,t)+(yt,x,z)+(zt,x,y),\]
where $ (x,y,z)=(xy)z - x(yz) $ is an \textit{associator}. It is easy to see that, $ G_x $ is $ 3 $-linear function and symmetric in every two variables.  
The following proposition was proved in \cite{He}:

\begin{proposition} Let $\mathfrak{C}$ be a commutative algebra. Then $ \mathfrak{C} $ is an almost-Jordan algebra if and only if
	\begin{equation}\label{GG}
	G_y(x,z,t)=G_x(y,z,t)
	\end{equation}
for all $ x,y,z,t\in\mathfrak{C}.$	
\end{proposition}

Now, let us set $ x=a,$  $ y=b,$ $ z=x$ and $ t=y $ in (\ref{GG}), then we have 
$G_{a}(b,x,y)=G_{b}(a,x,y)$. Therefore, by the definition of $ G_{x} $ we obtain the following identity:
\[(bx,a,y)+(by,a,x)+(xy,a,b)=(ax,b,y)+(ay,b,x)+(xy,b,a).\]

Then developing the associator we have 
$$((bx)a)y-(bx)(ay)+((by)a)x-(by)(ax)+((xy)a)b-(xy)(ab)=$$
$$((ax)b)y-(ax)(by)+((ay)b)x-(ay)(bx)+((xy)b)a-(xy)(ba).$$

Moreover, if $ \mathfrak{C} $ is a commutative algebra, then we obtain the following relation:
\begin{equation}
((xy)a)b+((xb)a)y+x((yb)a)=((xy)b)a+((xa)b)y+x((ya)b).
\end{equation}

This is the identity of what we call the $ \mathfrak{CCD} $-algebra.

\begin{corollary}
Let $ \mathfrak{C} $ be a commutative algebra. Then $  \mathfrak{C} $ is an almost-Jordan algebra if and only if $  \mathfrak{C} $ is a $ \mathfrak{CCD} $-algebra. \end{corollary}

 \subsection{Method of classification of nilpotent algebras}

Throughout this paper, we use the notations and methods well written in \cite{ha17,hac16,cfk18},
which we have adapted for the $\mathfrak{CCD}$ case with some modifications.
Further in this section we give some important definitions.

Let $({\bf A}, \cdot)$ be complex  $ \mathfrak{CCD} $-algebra
and $\mathbb V$ be a complex  vector space. The $\mathbb C$-linear space ${\rm Z^{2}}\left(
\bf A,\mathbb V \right) $ is defined as the set of all bilinear maps $\theta  \colon {\bf A} \times {\bf A} \longrightarrow {\mathbb V}$ such that
\[ \theta(x,y)=\theta(y,x) \]
\[ \theta((xy)a,b)+\theta((xb)a,y)+\theta(x,(yb)a)=\theta((xy)b,a)+\theta((xa)b,y)+\theta(x,(ya)b) \]

These elements will be called {\it cocycles}. For a
linear map $f$ from $\bf A$ to  $\mathbb V$, if we define $\delta f\colon {\bf A} \times
{\bf A} \longrightarrow {\mathbb V}$ by $\delta f  (x,y ) =f(xy )$, then one can check that $\delta f\in {\rm Z^{2}}\left( {\bf A},{\mathbb V} \right) $. We define ${\rm B^{2}}\left({\bf A},{\mathbb V}\right) =\left\{ \theta =\delta f\ : f\in {\rm Hom}\left( {\bf A},{\mathbb V}\right) \right\} $.
We define the {\it second cohomology space} ${\rm H^{2}}\left( {\bf A},{\mathbb V}\right) $ as the quotient space ${\rm Z^{2}}
\left( {\bf A},{\mathbb V}\right) \big/{\rm B^{2}}\left( {\bf A},{\mathbb V}\right) $.

\

Let $\operatorname{Aut}({\bf A}) $ be the automorphism group of  ${\bf A} $ and let $\phi \in \operatorname{Aut}({\bf A})$. For $\theta \in
{\rm Z^{2}}\left( {\bf A},{\mathbb V}\right) $ define  the action of the group $\operatorname{Aut}({\bf A}) $ on ${\rm Z^{2}}\left( {\bf A},{\mathbb V}\right) $ by $\phi \theta (x,y)
=\theta \left( \phi \left( x\right) ,\phi \left( y\right) \right) $.  It is easy to verify that
 ${\rm B^{2}}\left( {\bf A},{\mathbb V}\right) $ is invariant under the action of $\operatorname{Aut}({\bf A}).$
 So, we have an induced action of  $\operatorname{Aut}({\bf A})$  on ${\rm H^{2}}\left( {\bf A},{\mathbb V}\right)$.

\

Let $\bf A$ be a $ \mathfrak{CCD} $-algebra of dimension $m$ over  $\mathbb C$ and ${\mathbb V}$ be a $\mathbb C$-vector
space of dimension $k$. For the bilinear map $\theta$, define on the linear space ${\bf A}_{\theta } = {\bf A}\oplus {\mathbb V}$ the
bilinear product `` $\left[ -,-\right] _{{\bf A}_{\theta }}$'' by $\left[ x+x^{\prime },y+y^{\prime }\right] _{{\bf A}_{\theta }}=
 xy +\theta(x,y) $ for all $x,y\in {\bf A},x^{\prime },y^{\prime }\in {\mathbb V}$.
The algebra ${\bf A}_{\theta }$ is called a $k$-{\it dimensional central extension} of ${\bf A}$ by ${\mathbb V}$. One can easily check that ${\bf A_{\theta}}$ is a $ \mathfrak{CCD} $-algebra if and only if $\theta \in {\rm Z^2}({\bf A}, {\mathbb V})$.

Call the
set $\operatorname{Ann}(\theta)=\left\{ x\in {\bf A}:\theta \left( x, {\bf A} \right)=0\right\} $
the {\it annihilator} of $\theta $. We recall that the {\it annihilator} of an  algebra ${\bf A}$ is defined as
the ideal $\operatorname{Ann}(  {\bf A} ) =\left\{ x\in {\bf A}:  x{\bf A}=0\right\}$. Observe
 that
$\operatorname{Ann}\left( {\bf A}_{\theta }\right) =(\operatorname{Ann}(\theta) \cap\operatorname{Ann}({\bf A}))
 \oplus {\mathbb V}$.

\

The following result shows that every algebra with a non-zero annihilator is a central extension of a smaller-dimensional algebra.

\begin{lemma}
Let ${\bf A}$ be an $n$-dimensional $ \mathfrak{CCD} $-algebra such that $\dim (\operatorname{Ann}({\bf A}))=m\neq0$. Then there exists, up to isomorphism, a unique $(n-m)$-dimensional $ \mathfrak{CCD} $-algebra ${\bf A}'$ and a bilinear map $\theta \in {\rm Z^2}({\bf A'}, {\mathbb V})$ with $\operatorname{Ann}({\bf A'})\cap\operatorname{Ann}(\theta)=0$, where $\mathbb V$ is a vector space of dimension m, such that ${\bf A} \cong {{\bf A}'}_{\theta}$ and
 ${\bf A}/\operatorname{Ann}({\bf A})\cong {\bf A}'$.
\end{lemma}

\begin{proof}
Let ${\bf A}'$ be a linear complement of $\operatorname{Ann}({\bf A})$ in ${\bf A}$. Define a linear map $P \colon {\bf A} \longrightarrow {\bf A}'$ by $P(x+v)=x$ for $x\in {\bf A}'$ and $v\in\operatorname{Ann}({\bf A})$, and define a multiplication on ${\bf A}'$ by $[x, y]_{{\bf A}'}=P(x y)$ for $x, y \in {\bf A}'$.
For $x, y \in {\bf A}$, we have
\[P(xy)=P((x-P(x)+P(x))(y- P(y)+P(y)))=P(P(x) P(y))=[P(x), P(y)]_{{\bf A}'}. \]

Since $P$ is a homomorphism $P({\bf A})={\bf A}'$ is a $ \mathfrak{CCD} $-algebra and
 ${\bf A}/\operatorname{Ann}({\bf A})\cong {\bf A}'$, which gives us the uniqueness. Now, define the map $\theta \colon {\bf A}' \times {\bf A}' \longrightarrow\operatorname{Ann}({\bf A})$ by $\theta(x,y)=xy- [x,y]_{{\bf A}'}$.
  Thus, ${\bf A}'_{\theta}$ is ${\bf A}$ and therefore $\theta \in {\rm Z^2}({\bf A'}, {\mathbb V})$ and $\operatorname{Ann}({\bf A'})\cap\operatorname{Ann}(\theta)=0$.
\end{proof}

\

\begin{definition}
Let ${\bf A}$ be an algebra and $I$ be a subspace of $\operatorname{Ann}({\bf A})$. If ${\bf A}={\bf A}_0 \oplus I$
then $I$ is called an {\it annihilator component} of ${\bf A}$.
A central extension of an algebra $\bf A$ without annihilator component is called a {\it non-split central extension}.
\end{definition}

\ 

Our task is to find all central extensions of an algebra $\bf A$ by
a space ${\mathbb V}$.  In order to solve the isomorphism problem we need to study the
action of $\operatorname{Aut}({\bf A})$ on ${\rm H^{2}}\left( {\bf A},{\mathbb V}
\right) $. To do that, let us fix a basis $e_{1},\ldots ,e_{s}$ of ${\mathbb V}$, and $
\theta \in {\rm Z^{2}}\left( {\bf A},{\mathbb V}\right) $. Then $\theta $ can be uniquely
written as $\theta \left( x,y\right) =
\displaystyle \sum_{i=1}^{s} \theta _{i}\left( x,y\right) e_{i}$, where $\theta _{i}\in
{\rm Z^{2}}\left( {\bf A},\mathbb C\right) $. Moreover, $\operatorname{Ann}(\theta)=\operatorname{Ann}(\theta _{1})\cap\operatorname{Ann}(\theta _{2})\cap\ldots \cap\operatorname{Ann}(\theta _{s})$. Furthermore, $\theta \in
{\rm B^{2}}\left( {\bf A},{\mathbb V}\right) $\ if and only if all $\theta _{i}\in {\rm B^{2}}\left( {\bf A},
\mathbb C\right) $.
It is not difficult to prove (see \cite[Lemma 13]{hac16}) that given a $\mathfrak{CCD}$-algebra ${\bf A}_{\theta}$, if we write as
above $\theta \left( x,y\right) = \displaystyle \sum_{i=1}^{s} \theta_{i}\left( x,y\right) e_{i}\in {\rm Z^{2}}\left( {\bf A},{\mathbb V}\right) $ and
$\operatorname{Ann}(\theta)\cap \operatorname{Ann}\left( {\bf A}\right) =0$, then ${\bf A}_{\theta }$ has an
annihilator component if and only if $\left[ \theta _{1}\right] ,\left[
\theta _{2}\right] ,\ldots ,\left[ \theta _{s}\right] $ are linearly
dependent in ${\rm H^{2}}\left( {\bf A},\mathbb C\right) $.

Let ${\mathbb V}$ be a finite-dimensional vector space over $\mathbb C$. The {\it Grassmannian} $G_{k}\left( {\mathbb V}\right) $ is the set of all $k$-dimensional
linear subspaces of $ {\mathbb V}$. Let $G_{s}\left( {\rm H^{2}}\left( {\bf A},\mathbb C\right) \right) $ be the Grassmannian of subspaces of dimension $s$ in
${\rm H^{2}}\left( {\bf A},\mathbb C\right) $. There is a natural action of $\operatorname{Aut}({\bf A})$ on $G_{s}\left( {\rm H^{2}}\left( {\bf A},\mathbb C\right) \right) $.
Let $\phi \in \operatorname{Aut}({\bf A})$. For $W=\left\langle
\left[ \theta _{1}\right] ,\left[ \theta _{2}\right] ,\dots,\left[ \theta _{s}
\right] \right\rangle \in G_{s}\left( {\rm H^{2}}\left( {\bf A},\mathbb C
\right) \right) $ define $\phi W=\left\langle \left[ \phi \theta _{1}\right]
,\left[ \phi \theta _{2}\right] ,\dots,\left[ \phi \theta _{s}\right]
\right\rangle $. We denote the orbit of $W\in G_{s}\left(
{\rm H^{2}}\left( {\bf A},\mathbb C\right) \right) $ under the action of $\operatorname{Aut}({\bf A})$ by $\operatorname{Orb}(W)$. Given
\[
W_{1}=\left\langle \left[ \theta _{1}\right] ,\left[ \theta _{2}\right] ,\dots,
\left[ \theta _{s}\right] \right\rangle ,W_{2}=\left\langle \left[ \vartheta
_{1}\right] ,\left[ \vartheta _{2}\right] ,\dots,\left[ \vartheta _{s}\right]
\right\rangle \in G_{s}\left( {\rm H^{2}}\left( {\bf A},\mathbb C\right)
\right),
\]
we easily have that if $W_{1}=W_{2}$, then $ \bigcap\limits_{i=1}^{s}\operatorname{Ann}(\theta _{i})\cap \operatorname{Ann}\left( {\bf A}\right) = \bigcap\limits_{i=1}^{s}
\operatorname{Ann}(\vartheta _{i})\cap\operatorname{Ann}( {\bf A}) $, and therefore we can introduce
the set
\[
{\bf T}_{s}({\bf A}) =\left\{ W=\left\langle \left[ \theta _{1}\right] ,
\left[ \theta _{2}\right] ,\dots,\left[ \theta _{s}\right] \right\rangle \in
G_{s}\left( {\rm H^{2}}\left( {\bf A},\mathbb C\right) \right) : \bigcap\limits_{i=1}^{s}\operatorname{Ann}(\theta _{i})\cap\operatorname{Ann}({\bf A}) =0\right\},
\]
which is stable under the action of $\operatorname{Aut}({\bf A})$.

\

Now, let ${\mathbb V}$ be an $s$-dimensional linear space and let us denote by
${\bf E}\left( {\bf A},{\mathbb V}\right) $ the set of all {\it non-split $s$-dimensional central extensions} of ${\bf A}$ by
${\mathbb V}$. By above, we can write
\[
{\bf E}\left( {\bf A},{\mathbb V}\right) =\left\{ {\bf A}_{\theta }:\theta \left( x,y\right) = \sum_{i=1}^{s}\theta _{i}\left( x,y\right) e_{i} \ \ \text{and} \ \ \left\langle \left[ \theta _{1}\right] ,\left[ \theta _{2}\right] ,\dots,
\left[ \theta _{s}\right] \right\rangle \in {\bf T}_{s}({\bf A}) \right\} .
\]
We also have the following result, which can be proved as in \cite[Lemma 17]{hac16}.

\begin{lemma}
 Let ${\bf A}_{\theta },{\bf A}_{\vartheta }\in {\bf E}\left( {\bf A},{\mathbb V}\right) $. Suppose that $\theta \left( x,y\right) =  \displaystyle \sum_{i=1}^{s}
\theta _{i}\left( x,y\right) e_{i}$ and $\vartheta \left( x,y\right) =
\displaystyle \sum_{i=1}^{s} \vartheta _{i}\left( x,y\right) e_{i}$.
Then the $ \mathfrak{CCD} $-algebras ${\bf A}_{\theta }$ and ${\bf A}_{\vartheta } $ are isomorphic
if and only if
$$\operatorname{Orb}\left\langle \left[ \theta _{1}\right] ,
\left[ \theta _{2}\right] ,\dots,\left[ \theta _{s}\right] \right\rangle =
\operatorname{Orb}\left\langle \left[ \vartheta _{1}\right] ,\left[ \vartheta
_{2}\right] ,\dots,\left[ \vartheta _{s}\right] \right\rangle .$$
\end{lemma}

This shows that there exists a one-to-one correspondence between the set of $\operatorname{Aut}({\bf A})$-orbits on ${\bf T}_{s}\left( {\bf A}\right) $ and the set of
isomorphism classes of ${\bf E}\left( {\bf A},{\mathbb V}\right) $. Consequently we have a
procedure that allows us, given a $ \mathfrak{CCD} $-algebra ${\bf A}'$ of
dimension $n-s$, to construct all non-split central extensions of ${\bf A}'$. This procedure is:

\begin{enumerate}
\item For a given $ \mathfrak{CCD} $-algebra ${\bf A}'$ of dimension $n-s $, determine ${\rm H^{2}}( {\bf A}',\mathbb {C}) $, $\operatorname{Ann}({\bf A}')$ and $\operatorname{Aut}({\bf A}')$.

\item Determine the set of $\operatorname{Aut}({\bf A}')$-orbits on ${\bf T}_{s}({\bf A}') $.

\item For each orbit, construct the $ \mathfrak{CCD} $-algebra associated with a
representative of it.
\end{enumerate}

The above described method gives all (Jordan and non-Jordan) $ \mathfrak{CCD} $-algebras. But we are interested in developing this method in such a way that it only gives non-Jordan $ \mathfrak{CCD} $-algebras, because the classification of all Jordan algebras is given in \cite{kkk18}. Clearly, any central extension of a non-Jordan $ \mathfrak{CCD} $-algebra is non-Jordan. But a Jordan algebra may have extensions which are not Jordan algebras. More precisely, let $\mathfrak{J}$ be a Jordan algebra and $\theta \in {\rm Z_\mathfrak{CCD}^2}(\mathfrak{J}, {\mathbb C}).$ Then ${\mathfrak{J}}_{\theta }$ is a Jordan algebra if and only if
\[\theta(x,y)=\theta(y,x)\]
\[\theta(xy,zt)+\theta(xz,yt)+\theta(xt,yz)= \theta((xz)y,t)+\theta((zt)y,x)+\theta((tx)y,z).\]
for all $x,y,z,t\in {\mathfrak{J}}.$ Define the subspace ${\rm Z_\mathfrak{J}^2}({\mathfrak{J}},{\mathbb C})$ of ${\rm Z_\mathfrak{CCD}^2}({\bf  \mathfrak{J}},{\mathbb C})$ by
\begin{equation*}
{\rm Z_\mathfrak{J}^2}({\mathfrak{J}},{\mathbb C}) =\left\{\begin{array}{ll} \theta \in {\rm Z_\mathfrak{CCD}^2}({\mathfrak{J}},{\mathbb C}): & \theta(x,y)=\theta(y,x), \\  & \theta(xy,zt)+\theta(xz,yt)+\theta(xt,yz)= \\ 
& \theta((xz)y,t)+\theta((zt)y,x)+\theta((tx)y,z)\\
& \text{ for all } x, y,z,t\in {\mathfrak{J}}\end{array}\right\}.
\end{equation*}

Observe that ${\rm B^2}({ \mathfrak{J}},{\mathbb C})\subseteq{\rm Z_\mathfrak{J}^2}({\mathfrak{J}},{\mathbb C}).$
Let ${\rm H_\mathfrak{J}^2}({\mathfrak{J}},{\mathbb C}) =%
{\rm Z_\mathfrak{J}^2}({\mathfrak{J}},{\mathbb C}) \big/{\rm B^2}({\mathfrak{J}},{\mathbb C}).$ Then ${\rm H_\mathfrak{J}^2}({\mathfrak{J}},{\mathbb C})$ is a subspace of $%
{\rm H_\mathfrak{CCD}^2}({\mathfrak{J}},{\mathbb C}).$ Define
\[{\bf R}_{s}({\mathfrak{J}}) =\left\{ {\bf W}\in {\bf T}_{s}({\mathfrak{J}}) :{\bf W}\in G_{s}({\rm H_\mathfrak{J}^2}({\mathfrak{J}},{\mathbb C}) ) \right\}, \]
\[	{\bf U}_{s}({\mathfrak{J}}) =\left\{ {\bf W}\in {\bf T}_{s}({\mathfrak{J}}) :{\bf W}\notin G_{s}({\rm H_\mathfrak{J}^2}({\mathfrak{J}},{\mathbb C}) ) \right\}.\]

Then ${\bf T}_{s}({\mathfrak{J}}) ={\bf R}_{s}(
{\mathfrak{J}})$ $\mathbin{\mathaccent\cdot\cup}$ ${\bf U}_{s}(
{\mathfrak{J}}).$ The sets ${\bf R}_{s}({\mathfrak{J}}) $
and ${\bf U}_{s}({\mathfrak{J}})$ are stable under the action
of $\operatorname{Aut}({\mathfrak{J}}).$ Thus, the $ \mathfrak{CCD} $-algebras
corresponding to the representatives of $\operatorname{Aut}({\mathfrak{J}}) $%
-orbits on ${\bf R}_{s}({\mathfrak{J}})$ are Jordan  algebras,
while those corresponding to the representatives of $\operatorname{Aut}({\mathfrak{J}}%
) $-orbits on ${\bf U}_{s}({\mathfrak{J}})$ are non-Jordan algebras. Hence, we may construct all non-split non-Jordan $ \mathfrak{CCD} $-algebras $%
\bf{A}$ of dimension $n$ with $s$-dimensional annihilator
from a given $ \mathfrak{CCD} $-algebra $\bf{A}%
^{\prime }$ of dimension $n-s$ in the following way:

\begin{enumerate}
\item If $\bf{A}^{\prime }$ is non-Jordan, then apply the procedure.

\item Otherwise, do the following:

\begin{enumerate}
\item Determine ${\bf U}_{s}(\bf{A}^{\prime })$ and $%
\operatorname{Aut}(\bf{A}^{\prime }).$

\item Determine the set of $\operatorname{Aut}(\bf{A}^{\prime })$-orbits on ${\bf U%
}_{s}(\bf{A}^{\prime }).$

\item For each orbit, construct the $ \mathfrak{CCD} $-algebra corresponding to one of its
representatives.
\end{enumerate}
\end{enumerate}

\subsection{Notations}
Let us introduce the following notations. Let ${\bf A}$ be a nilpotent algebra with
a basis $e_{1},e_{2}, \ldots, e_{n}.$ Then by $\Delta_{ij}$\ we will denote the commutative
bilinear form
$\Delta_{ij}:{\bf A}\times {\bf A}\longrightarrow \mathbb C$
$\Delta_{ij}(e_{l},e_{m})  = \delta_{il}\delta_{jm},$
if $i\leq j$ and $l\leq m.$
The set $\left\{ \Delta_{ij}:1\leq i\leq j\leq n\right\}$ is a basis for the linear space of
bilinear forms on ${\bf A},$ so every $\theta \in
{\rm Z^2}({\bf A},\bf \mathbb V )$ can be uniquely written as $
\theta = \displaystyle \sum_{1\leq i\leq j\leq n} c_{ij}\Delta _{ij}$, where $
c_{ij}\in \mathbb C$.
Let us fix the following notations:

$$\begin{array}{lll}
\mathfrak{C}^{i}_{j}& \mbox{---}& j\mbox{th }i\mbox{-dimensional $\mathfrak{CCD}$  (non-Jordan)  algebra.} \\
\mathfrak{C}^{i*}_{j}& \mbox{---}& j\mbox{th }i\mbox{-dimensional $\mathfrak{CCD}$  (Jordan) algebra.}
\end{array}$$

All algebras from the present paper are commutative and considered over the complex field. 
 
\subsection{$4$-dimensional $\mathfrak{CCD}$-algebras}
Thanks to \cite{cfk18} we have an algebraic classification of all  complex $4$-dimensional nilpotent
$\mathfrak{CCD}$-algebras with $2$- and $3$-dimensional annihilator:

\begin{longtable}{llllllll}

			$\mathfrak{C}^{4*}_{01}$&$:$& $e_1 e_1 = e_2$\\
			
			\hline
			$\mathfrak{C}^{4*}_{02}$&$:$& $e_1 e_1 = e_2$ & $e_1 e_2=e_3$\\

			\hline
            $\mathfrak{C}^{4*}_{03}$&$:$& $e_1 e_2=e_3$\\
        	\hline

			$\mathfrak{C}^{4*}_{04}$&$:$& $e_1 e_1 = e_3$  & $e_2 e_2=e_4$\\

			\hline
			$\mathfrak{C}^{4*}_{05}$&$:$& $e_1 e_1 = e_3$ & $e_1 e_2=e_4$\\

			\hline
			$\mathfrak{C}^{4*}_{06}$&$:$& $e_1 e_1 = e_4$ & $e_2 e_3=e_4$\\
			\hline
        	$\mathfrak{C}^{4}_{01}$&$:$& $e_1 e_1 = e_2$ & $e_2 e_2=e_3$\\
			            \hline
    		$\mathfrak{C}^{4}_{02}(0)$  &$:$&  $e_1 e_1 = e_2$ & $e_1 e_2=e_4$ & $e_2 e_2=e_3$

		\end{longtable}
\medskip

\section{Central extensions of  $3$-dimensional nilpotent $\mathfrak{CCD}$-algebras}

\subsection{2-dimensional central extensions of ${\mathfrak{C}_{01}^{3*}}$}
Here we will  collect all information about ${\mathfrak{C}_{01}^{3*}}:$
\begin{longtable}{|l|l|l|l|}
\hline
Algebra  & Multiplication & Cohomology &  Automorphisms  \\
\hline
${\mathfrak{C}}_{01}^{3*}$ &  $  e_1 e_1 = e_2 $ & 
$\begin{array}{lcl}
{\rm H^2_{\mathfrak{J}}}({\mathfrak{C}}_{01}^{3*})&=&\langle [\Delta_{12}],[\Delta_{13}],[\Delta_{23}],[\Delta_{33}] \rangle \\
{\rm H^2_{\mathfrak{CCD}}}({\mathfrak{C}}_{01}^{3*})&=&{\rm H^2_{\mathfrak{J}}}({\mathfrak{C}_{01}^{3*}}) \oplus \langle[\Delta_{22}]\rangle 
\end{array}$& 
$\phi=
	\begin{pmatrix}
	x &    0  &  0\\
	y &  x^2  &  u\\
	z &   0  &  v
	\end{pmatrix}
	$\\
\hline
\end{longtable}

Let us use the following notations
	\begin{align*}
	\nabla_1 = [\Delta_{12}] \quad \nabla_2 = [\Delta_{13}] \quad \nabla_3 = [\Delta_{23}] \quad \nabla_4 = [\Delta_{33}] \quad \nabla_5 = [\Delta_{22}].
	\end{align*}
	Take $\theta=\sum\limits_{i=1}^5\alpha_{i}\nabla_{i}\in {\rm H_{\mathfrak{CD}}^2}(\mathfrak{C}_{01}^{3*}).$
	Since
 	$$
	\phi^T\begin{pmatrix}
	0      &  \alpha_1  & \alpha_2\\
	\alpha_1  & \alpha_5 & \alpha_3\\
	\alpha_2 &  \alpha_3   & \alpha_4
	\end{pmatrix} \phi=
	\begin{pmatrix}
	\alpha^*      &  \alpha^{*}_{1}    & \alpha^*_2\\
	\alpha^*_1    & \alpha^*_5    & \alpha^*_3\\
	\alpha^*_2  &  \alpha^*_3    & \alpha^*_4
	\end{pmatrix},
	$$
	we have
	\begin{longtable}{l}
	$\alpha^*_1 = (\alpha_{1}x+\alpha_{3}z+\alpha_{5}y)x^2$\\
	$\alpha^*_2 = (\alpha_{1}x+\alpha_{3}z+\alpha_{5}y)u+(\alpha_{2}x+\alpha_{3}y+\alpha_{4}z)v$\\
	$\alpha^*_3 = (\alpha_{5}u+\alpha_{3}v)x^2$\\
	$\alpha^*_4 = 2\alpha_{3}uv+\alpha_{4}v^2+\alpha_{5}u^2$\\
	$\alpha^*_5 = \alpha_{5}x^4.$\\
	\end{longtable}

\subsubsection{$1$-dimensional central extensions of ${\mathfrak{C}_{01}^{3*}}$}
Thanks to \cite{fkkv}, we have the following $4$-dimensional $\mathfrak{CCD}$-algebras:

\begin{longtable}{llllllllllllllllllllll}
$\mathfrak{C}^{4*}_{07}$ & $e_1 e_1 = e_2$ & $e_2e_3=e_4$ \\
\hline
$\mathfrak{C}^{4*}_{08}$&  $e_1 e_1 = e_2$& $e_1e_2=e_4$&$ e_3e_3=e_4$ \\
\hline
$\mathfrak{C}^{4}_{03}$&   $e_1 e_1 = e_2$& $e_1e_3=e_4$& $e_2e_2=e_4$		\\
			
			\hline
$\mathfrak{C}^{4}_{04}$&     $e_1 e_1 = e_2$ & $e_2e_2=e_4$& $e_3e_3=e_4$ 
\end{longtable}

\subsubsection{$2$-dimensional central extensions of ${\mathfrak{C}_{01}^{3*}}$}
Since we are interested only in new algebras, consider the vector space generated by the following two cocycles:
$$\begin{array}{rclcrcl}
\theta_1 &=& \sum\limits_{i=1}^4 \alpha_i \nabla_i + \nabla_5  &
\mbox{ and }
&\theta_2 &=& \sum\limits_{i=1}^4 \beta_i \nabla_i.
\end{array}$$

Thus, we have
\begin{longtable}{ll}
	$\alpha^*_1 = (\alpha_{1}x+\alpha_{3}z+y)x^2$ & $\beta^*_1 = (\beta_{1}x+\beta_{3}z)x^2$\\
	$\alpha^*_2 = (\alpha_{1}x+\alpha_{3}z+y)u+(\alpha_{2}x+\alpha_{3}y+\alpha_{4}z)v$& $\beta^*_2 = (\beta_{1}x+\beta_{3}z)u+(\beta_{2}x+\beta_{3}y+\beta_4 z)v$\\
	$\alpha^*_3 = (\alpha_{3}v+u)x^2$ & $\beta^*_3 = \beta_{3}vx^2$\\
	$\alpha^*_4 = 2\alpha_{3}uv+\alpha_{4}v^2+u^2$ & $\beta^*_4 = 2\beta_{3}uv+\beta_{4}v^2$\\
	$\alpha^*_5 = x^4.$& 
	\end{longtable}

Then we have the following cases:
\begin{enumerate} \item if $\beta_3=0, \beta_4=0, \beta_1=0,$ then
   \begin{enumerate}
   \item if $ \alpha_4=\alpha_{3}^2,$ then choosing  $x=1,$ $z=0,$ $y = -\alpha_1,$  $v=1$ and $u = - \alpha_3,$ we have the representative $\langle \nabla_5, \nabla_2 \rangle;$
   \item if $ \alpha_4\neq\alpha_{3}^2,$ then choosing  $x=1,$ $z=0,$ $y = -\alpha_1,$  $v=\frac{1}{ \sqrt{\alpha_4-\alpha_3^2}}$ and $u = - \alpha_3,$ we have the representative $\langle \nabla_4+\nabla_5,\nabla_2 \rangle.$

   \end{enumerate}
   \item if $\beta_3=0, \beta_4=0, \beta_1\neq0,$ then
   \begin{enumerate}
   \item if $\beta_2=\beta_1\alpha_3, \alpha_4=\alpha_{3}^2, \alpha_2=\alpha_1\alpha_3, $ then choosing  $x=1,$ $z=0,$ $y = -\alpha_1,$ $v=1$ and $u = - \alpha_3,$ we have the representative $\langle \nabla_5,\nabla_1 \rangle;$
   \item if $\beta_2=\beta_1\alpha_3, \alpha_4=\alpha_{3}^2, \alpha_2\neq \alpha_1\alpha_3, $ then choosing  
   $x=1,$ $z=0,$  $y = -\alpha_1,$  $v = \frac{1}{\alpha_2-\alpha_1\alpha_3}$ and $u = - v\alpha_3,$
     we have the representative $\langle \nabla_2+\nabla_5,\nabla_1 \rangle;$
   \item if $\beta_2=\beta_1\alpha_3, \alpha_4\neq\alpha_{3}^2, $ then choosing  
   $x=1,$ $z = \frac{(\alpha_2-\alpha_1\alpha_3)}{\alpha_3^2-\alpha_4},$ $y = -\alpha_1 - z \alpha_3,$
       $v = \frac{1}{\sqrt{\alpha_4-\alpha_3^2}}$ and $u = -v\alpha_3,$  we have the representative $\langle \nabla_4+\nabla_5,\nabla_1 \rangle;$
   \item if $\beta_2\neq\beta_1\alpha_3, $ then choosing $x=1,$ $z=0,$ $y = \frac{(\alpha_2\beta_1-\alpha_1\beta_2)}{\beta_2-\beta_1\alpha_3},$  $v = \frac{\beta_1}{\beta_1\alpha_3-\beta_2}$ and 
           $u = -\frac{v\beta_2}{\beta_1},$ we have the family of representatives $\langle \nabla_3 +\alpha \nabla_4+\nabla_5, \nabla_1 \rangle.$
   \end{enumerate}

 \item if $\beta_3=0, \beta_4\neq0,$ then

   \begin{enumerate}
   \item  if $\beta_1=0,$ $\beta_4(\alpha_2-\alpha_1\alpha_3) = \beta_2(\alpha_4-\alpha_3^2),$
   then  choosing $x=1,$ $z=-\frac{\beta_2}{\beta_{4}},$ $y = - \alpha_1 - z \alpha_3,$ $v=1,$ $u = -\alpha_3,$   we have the representative $\langle \nabla_5, \nabla_4\rangle;$
   \item  if $\beta_1=0,$ $\beta_4(\alpha_2-\alpha_1\alpha_3) \neq \beta_2(\alpha_4-\alpha_3^2),$  then
    choosing $x=1,$ $z=-\frac{\beta_2}{\beta_{4}},$ $y = - \alpha_1 - z \alpha_3,$ $v=\frac{\beta_4}{\beta_4(\alpha_2-\alpha_1\alpha_3) - \beta_2(\alpha_4-\alpha_3^2)}$ and $u = -v \alpha_3,$ we have the representative  $\langle \nabla_2+\nabla_5,\nabla_4 \rangle;$
    \item if $\beta_1\neq 0,$ $\beta_4(\alpha_2-\alpha_1\alpha_3) = (\alpha_4-\alpha_3^2)(\beta_2-\beta_1\alpha_3),$ then
    choosing $x=1,$ $z=-\frac{\beta_2-\beta_1\alpha_3}{\beta_4},$ $y=\frac{\beta_1\alpha_4+\beta_2\alpha_3-\beta_4\alpha_1-2\beta_1\alpha^2_3}{\beta_4},$  $v=\sqrt{\frac{\beta_1}{\beta_4}}$ and 
     $u = -v \alpha_3,$ we have the representative  $\langle -\nabla_5,\nabla_1+\nabla_4 \rangle;$
    \item if $\beta_1\neq 0,$ $\beta_4(\alpha_2-\alpha_1\alpha_3) \neq (\alpha_4-\alpha_3^2)(\beta_2-\beta_1\alpha_3),$ then
    choosing $x=\sqrt[3]{\frac{\beta_1(\beta_4(\alpha_2-\alpha_1\alpha_3) - (\alpha_4-\alpha_3^2)(\beta_2-\beta_1\alpha_3))}{\beta_4}},$
    $z=-\frac{x(\beta_2-\beta_1\alpha_3)}{\beta_4},$ $y=\frac{x(\beta_1\alpha_4+\beta_2\alpha_3-\beta_4\alpha_1-2\beta_1\alpha^2_3)}{\beta_4},$
      $v=\frac{x^3}{\beta_4(\alpha_2-\alpha_1\alpha_3) - (\alpha_4-\alpha_3^2)(\beta_2-\beta_1\alpha_3)}$  and
    $u = -v \alpha_3,$ we have the representative  $\langle \nabla_2+\nabla_5,\nabla_1+\nabla_4 \rangle.$

\end{enumerate}

   \item if $\beta_3\neq0,$ then choosing $z=-\frac{x\beta_1}{\beta_{3}},$ $y=\frac{x(\beta_1\beta_4-\beta_2\beta_3)}{\beta_{3}^2},$  $u=-\frac{v\beta_4}{2\beta_{3}},$ we have $\beta_1^*=\beta_2^*=\beta_4^*=0$ and
       \begin{longtable}{l}
       $\alpha_1^* = \frac {\alpha_1\beta_3^2-\alpha_3\beta_1\beta_3+\beta_1\beta_4-\beta_2\beta_3}{\beta_3^2} x^3$\\
       $\alpha_2^* = \frac {-\beta_4(\alpha_1\beta_3^2-\alpha_3\beta_1\beta_3+\beta_1\beta_4-\beta_2\beta_3) +2\beta_3(\alpha_2\beta_3^2+\alpha_3\beta_1\beta_4-\alpha_3\beta_2\beta_3-\alpha_4\beta_1\beta_3)}{2\beta_3^3} x v$ \\
       $\alpha_3^* = \frac {2\alpha_3\beta_3-\beta_4}{2\beta_3} x^2v$\\
       $\alpha_4^* = \frac {-4\alpha_3\beta_3\beta_4+4\alpha_4\beta_3^2+\beta_4^2}{4\beta_3^2} v^2$\\
       $\alpha_5^*= x^4.$
       \end{longtable}
       Let us give the denotation
       \begin{longtable}{l}  $A=\alpha_1\beta_3^2-\alpha_3\beta_1\beta_3+\beta_1\beta_4-\beta_2\beta_3$\\
       $B=-\beta_4(\alpha_1\beta_3^2-\alpha_3\beta_1\beta_3+\beta_1\beta_4-\beta_2\beta_3) +2\beta_3(\alpha_2\beta_3^2+\alpha_3\beta_1\beta_4-\alpha_3\beta_2\beta_3-\alpha_4\beta_1\beta_3)$\\
       $C=-4\alpha_3\beta_3\beta_4+4\alpha_4\beta_3^2+\beta_4^2.$
       \end{longtable}
Consider the following subcases:
   \begin{enumerate}
   \item if $ A=0,$
   $B=0,$
   $C=0,$
   then we have the representative $\langle \nabla_5, \nabla_3 \rangle$;
   \item if $ A=0,$
   $B=0,$
   $C\neq 0,$
   then choosing $x=1,$ $v=\frac{2\beta_3}{\sqrt{C}},$ we have the representative $\langle \nabla_4 + \nabla_5, \nabla_3 \rangle$;
   \item if $ A=0,$
   $B\neq0,$
   $C= 0,$
   then choosing $ v=\frac{2\beta^3_3}{B}x^3,$ we have the representative $\langle \nabla_2 + \nabla_5, \nabla_3 \rangle$;
   \item if $ A=0,$
   $B\neq0,$
   $C\neq 0,$
   then choosing $x=\frac{B}{\beta^2_3\sqrt{C}},$ $v=\frac{2B}{\beta_3 C}x,$ we have the representative $\langle \nabla_2 + \nabla_4 + \nabla_5, \nabla_3 \rangle$;
   \item if $ A\neq 0,$
   $B=0,$
   $C= 0,$
   then choosing $ x=\frac{A} {\beta^2_3},$ we have the representative $\langle \nabla_1 + \nabla_5, \nabla_3 \rangle$;
   \item if $ A\neq 0,$
   $B=0,$
   $C\neq 0,$
   then choosing $x=\frac{A}{\beta^2_3},$ $ v=\frac{2\beta_3}{\sqrt{C}}x^2,$ we have the representative $\langle \nabla_1 + \nabla_4 + \nabla_5, \nabla_3 \rangle$;
   \item if $ A\neq 0,$
   $B\neq 0,$
      then choosing $x=\frac{A}{\beta^2_3},$ $ v=\frac{2\beta_3 A}{B}x^2,$ we have the family of representatives $\langle \nabla_1 + \nabla_2 + \alpha \nabla_4 + \nabla_5, \nabla_3 \rangle$;

   \end{enumerate}
\end{enumerate}

Now we have the following distinct orbits:

\begin{longtable} {lll}
$\langle \nabla_5, \nabla_2 \rangle$ & $\langle \nabla_4+\nabla_5,\nabla_2 \rangle$ & $\langle \nabla_5,\nabla_1 \rangle$ \\
$\langle \nabla_2+\nabla_5,\nabla_1 \rangle$ & $\langle \nabla_4+\nabla_5,\nabla_1 \rangle$ &
$\langle \nabla_3 +\alpha \nabla_4+\nabla_5, \nabla_1 \rangle$ \\
$\langle \nabla_5, \nabla_4\rangle$ &
$\langle \nabla_2+\nabla_5,\nabla_4 \rangle$ &
$\langle -\nabla_5,\nabla_1+\nabla_4 \rangle$ \\
$\langle \nabla_2+\nabla_5,\nabla_1+\nabla_4 \rangle$& $\langle \nabla_5, \nabla_3 \rangle$ & $\langle \nabla_4 + \nabla_5, \nabla_3 \rangle$ \\
$\langle \nabla_2 + \nabla_5, \nabla_3 \rangle$ & $\langle \nabla_2 + \nabla_4 + \nabla_5, \nabla_3 \rangle$ & $\langle \nabla_1 + \nabla_5, \nabla_3 \rangle$ \\
$\langle \nabla_1 + \nabla_4 + \nabla_5, \nabla_3 \rangle$ & $\langle \nabla_1 + \nabla_2 + \alpha \nabla_4 + \nabla_5, \nabla_3 \rangle$
\end{longtable}

Hence, we have the following new $5$-dimensional nilpotent $\mathfrak{CCD}$-algebras:

\begin{longtable}{llllllllll}
$\mathfrak{C}^5_{08}$ & $: $ & $e_1e_1=e_2$ & $e_1e_3=e_4$ & $e_2e_2=e_5 $\\
$\mathfrak{C}^5_{09}$ & $: $& $e_1e_1=e_2$ & $e_1e_3=e_4$ & $e_2e_2=e_5$ & $e_3e_3=e_5 $\\
$\mathfrak{C}^5_{10}$ & $ :$ & $e_1e_1=e_2$ & $e_1e_2=e_4$ & $e_2e_2=e_5 $\\
$\mathfrak{C}^5_{11}$ & $ : $ & $e_1e_1=e_2$ & $e_1e_2=e_4$ & $e_1e_3=e_5$ & $e_2e_2=e_5 $\\
$\mathfrak{C}^5_{12}(-1)$ & $ : $ & $e_1e_1=e_2$ & $e_1e_2=e_4$ & $e_2e_2=e_5$ & $e_3e_3=e_5 $\\
$\mathfrak{C}^5_{13}(-1,\beta)$ & $ : $ & $e_1e_1=e_2$ & $e_1e_2=e_4$ & $e_2e_2=e_5$ & $e_2e_3=e_5$ & $e_3e_3=\beta e_5 $\\
$\mathfrak{C}^5_{14}$ & $ : $ & $e_1e_1=e_2$ & $e_2e_2=e_5$ & $e_3e_3=e_4 $\\
$\mathfrak{C}^5_{15}$ & $ : $ & $e_1e_1=e_2$ & $e_1e_3=e_5 $ & $ e_2e_2=e_5$ & $e_3e_3=e_4 $\\
$\mathfrak{C}^5_{16}(-1)$ & $ : $ & $e_1e_1=e_2$ & $e_1e_2=e_4$ & $e_2e_2=-e_5$ & $e_3e_3=e_4 $\\
$\mathfrak{C}^5_{17}$ & $ : $ & $e_1e_1=e_2$ & $e_1e_2=e_4$ & $e_1e_3=e_5$ & $e_2e_2=e_5$ & $e_3e_3=e_4 $\\
$\mathfrak{C}^5_{18}$ & $ : $ & $e_1e_1=e_2$ & $e_2e_2=e_5$ & $e_2e_3=e_4 $\\
$\mathfrak{C}^5_{19}$ & $ : $ & $e_1e_1=e_2$ & $e_2e_2=e_5$ & $e_2e_3=e_4$ & $e_3e_3=e_5 $\\
$\mathfrak{C}^5_{20}$ & $ : $ & $e_1e_1=e_2$ & $e_1e_3=e_5$ & $e_2e_2=e_5$ & $e_2e_3=e_4 $\\
$\mathfrak{C}^5_{21}$ & $ : $ & $e_1e_1=e_2$ & $e_1e_3=e_5$ & $e_2e_2=e_5$ & $e_2e_3=e_4$ & $e_3e_3=e_5 $\\
$\mathfrak{C}^5_{22}$ & $ : $ & $e_1e_1=e_2$ & $e_1e_2=e_5$ & $e_2e_2=e_5$ & $e_2e_3=e_4 $\\
$\mathfrak{C}^5_{23}$ & $ : $ & $e_1e_1=e_2$ & $e_1e_2=e_5$ & $e_2e_2=e_5$ & $e_2e_3=e_4$ & $e_3e_3=e_5 $\\
$\mathfrak{C}^5_{24}(\alpha)$ & $ : $ & $e_1e_1=e_2$ & $e_1e_2=e_5$ & $e_1e_3=e_5$ & $e_2e_2=e_5$ & $e_2e_3=e_4$ & $e_3e_3=\alpha e_5 $

\end{longtable} 

 \subsection{2-dimensional central extensions of ${\mathfrak{C}_{02}^{3*}}$}
Here we will  collect all information about ${\mathfrak{C}_{02}^{3*}}:$
\begin{longtable}{|l|l|l|l|}
\hline
Algebra  & Multiplication & Cohomology   \\
\hline
${\mathfrak{C}_{02}^{3*}}$ &  $ e_1 e_1 = e_2, e_1 e_2 = e_3 $ & 
$\begin{array}{lcl}
{\rm H^2_{\mathfrak{J}}}({\mathfrak{C}}_{02}^{3*})&=&\langle [\Delta_{13}+\Delta_{22}]\rangle \\
{\rm H^2_{\mathfrak{CCD}}}({\mathfrak{C}}_{02}^{3*})&=&{\rm H^2_{\mathfrak{J}}}({\mathfrak{C}_{02}^{3*}}) \oplus \langle[\Delta_{22}]\rangle 
\end{array}$\\
\hline
\end{longtable}
Hence, we have following new algebra:
\begin{longtable}{llllllllll}
$\mathfrak{C}^{4*}_{09}$ &$:$&  $e_1 e_1 = e_2$ & $e_1 e_2=e_3$ & $e_1e_3=e_4$  & $e_2e_2= e_4$ \\
\hline
$\mathfrak{C}^{4}_{02}(\alpha)$&$:$&  $e_1 e_1 = e_2$  & $e_1 e_2=e_3$& $e_1e_3= \alpha e_4$  & $e_2e_2= (\alpha +1)e_4$ \\
\hline
$\mathfrak{C}_{25}^5$&$:$& $e_1e_1=e_2$ & $e_1e_2=e_3$ &$e_1e_3=e_4$ & $e_2e_2=e_5$
\end{longtable}

\subsection{2-dimensional central extensions of ${\mathfrak{C}_{03}^{3*}}$}
Here we will  collect all information about ${\mathfrak{C}_{03}^{3*}}:$
{\tiny \begin{longtable}{|l|l|l|l|}
\hline
Algebra  & Multiplication & Cohomology &  Automorphisms  \\
\hline

${\mathfrak{C}_{03}^{3*}}$ &  $ e_1 e_2 = e_3 $ &
$\begin{array}{lcl}
{\rm H^2_{\mathfrak{J}}}({\mathfrak{C}_{03}^{3*}})&=&\langle [\Delta_{11}],[\Delta_{22}],[\Delta_{13}],[\Delta_{23}]\rangle \\
{\rm H^2_{\mathfrak{CCD}}}({\mathfrak{C}_{03}^{3*}})&=&{\rm H^2_{\mathfrak{J}}}({\mathfrak{C}_{03}^{3*}}) \oplus \langle[\Delta_{33}]\rangle
\end{array}$&
$\phi_1=
	\begin{pmatrix}
	x &    0  &  0\\
	0 &  u  &  0\\
	z &  v  &  xu
	\end{pmatrix},
	\phi_2=
	\begin{pmatrix}
	0 & u  &  0\\
	x &  0  &  0\\
	z &  v  &  xu
	\end{pmatrix}$
	\\
\hline
\end{longtable}}	
Let us use the following notations
\begin{align*}
	\nabla_1 = [\Delta_{11}] \quad \nabla_2 = [\Delta_{22}] \quad \nabla_3 = [\Delta_{13}]  \quad \nabla_4 = [\Delta_{23}] \quad \nabla_5 = [\Delta_{33}].
	\end{align*}
	Take $\theta=\sum\limits_{i=1}^5\alpha_{i}\nabla_{i}\in {\rm H_{\mathfrak{CD}}^2}(\mathfrak{C}_{03}^{3*}).$
	Since
	$$
	\phi_1^T\begin{pmatrix}
	 \alpha_1    & 0  & \alpha_3\\
	0 & \alpha_2 & \alpha_4\\
	\alpha_3 &  \alpha_4   & \alpha_5
	\end{pmatrix} \phi_1=
	\begin{pmatrix}
	\alpha^*_1      &  \alpha^*    & \alpha^*_3\\
	\alpha^*    & \alpha^*_2    & \alpha^*_4\\
	\alpha^*_3  &  \alpha^*_4    & \alpha^*_5
	\end{pmatrix},
	$$
we have
\begin{longtable}{l}
	$\alpha^*_1 = x^2 \alpha_1+2 x z \alpha_3+z^2 \alpha_5$\\
	$\alpha^*_2 = u^2 \alpha_2+2 u v \alpha_4+v^2 \alpha_5$\\
	$\alpha^*_3 = x u (x \alpha_3+z \alpha_5)$\\
	$\alpha^*_4 = x u (u \alpha_4+v \alpha_5)$\\
	$\alpha^*_5 = x^2u^2 \alpha_5$.\\
	\end{longtable}

\subsubsection{$1$-dimensional central extensions of ${\mathfrak{C}_{03}^{3*}}$}
Thanks to \cite{fkkv}, we have the following $4$-dimensional $\mathfrak{CCD}$-algebras:

\begin{longtable}{llllllllllllllllllllll}
$\mathfrak{C}^{4*}_{10}$& $e_1 e_2 = e_3$ & $e_1 e_3=e_4$& $e_2e_2=e_4$ \\ 	\hline
$\mathfrak{C}^{4*}_{11}$& $e_1 e_2 = e_3$ & $e_1 e_3=e_4$ & $e_2e_3=e_4$ \\\hline
$\mathfrak{C}^{4*}_{12}$& $e_1 e_2 = e_3$ & $e_1 e_3=e_4$ \\\hline
$\mathfrak{C}^{4}_{05}$ & $e_1 e_2 = e_3$ & $e_3 e_3=e_4$	\\ \hline
$\mathfrak{C}^{4}_{06}$ & $e_1 e_1 = e_4$ & $e_1 e_2=e_3$ & $e_2e_2=e_4$& $e_3e_3=e_4$\\\hline
$\mathfrak{C}^{4}_{07}$ & $e_1 e_1 = e_4$ & $e_1 e_2=e_3$ & $e_3e_3=e_4$\\
\end{longtable}

\subsubsection{$2$-dimensional central extensions of ${\mathfrak{C}_{03}^{3*}}$}
Consider the vector space generated by the following two cocycles:
$$\begin{array}{rclcrcl}
\theta_1 &=& \sum\limits_{i=1}^4 \alpha_i \nabla_i + \nabla_5  &
\mbox{ and }
&\theta_2 &=& \sum\limits_{i=1}^4 \beta_i \nabla_i,
\end{array}$$

Thus, we have
\begin{longtable}{ll}
	$\alpha^*_1 = x^2 \alpha_1+2 x z \alpha_3+z^2 $ & $\beta^*_1 = x^2 \beta_1+2 x z \beta_3$\\
	$\alpha^*_2 = u^2 \alpha_2+2 u v \alpha_4+v^2 $ & $\beta^*_2 = u^2 \beta_2+2 u v \beta_4$\\
	$\alpha^*_3 = x u (x \alpha_3+z )$ & $\beta^*_3 = x^2 u \beta_3$\\
	$\alpha^*_4 = x u (u \alpha_4+v )$ &$\beta^*_4 = x u^2 \beta_4$\\
	$\alpha^*_5 = x^2u^2$
	\end{longtable}

Then we consider the following cases:
\begin{enumerate}
    \item if $\beta_3\neq 0,$ $\beta_4\neq 0,$ then we may suppose $\alpha_4=0,$ and

\begin{enumerate}
    \item  if    $2\alpha_3 \beta_3 \beta_4^2 \neq \beta_1\beta_4^2 - \beta_2 \beta_3^2,$
    then choosing 
    $x=\frac{2\alpha_3 \beta_3 \beta_4^2-\beta_1\beta_4^2 +\beta_2 \beta_3^2}{2\beta_3^2\beta_4},$
    $z=-\frac{x\beta_1}{2\beta_3},$ $u=\frac{x\beta_3}{\beta_4}$ and $v=-\frac{u\beta_2}{2\beta_4},$ we have the family of representatives
    $\langle \alpha \nabla_1+\beta \nabla_2+\nabla_3+\nabla_5, \nabla_3+\nabla_4 \rangle;$
\item    if $2\alpha_3 \beta_3 \beta_4^2 = \beta_1\beta_4^2 - \beta_2 \beta_3^2,$ and $4\alpha_2\beta_4^2\neq - \beta_2^2,$
    then choosing 
    $x=\frac{\sqrt{4\alpha_2\beta_4^2+\beta_2^2}}{2\beta_4},$ $z=-\frac{x\beta_1}{2\beta_3},$ $u=\frac{x\beta_3}{\beta_4}$ and $v=-\frac{u\beta_2}{2\beta_4},$
     we have the family of representatives
    $\langle \alpha \nabla_1+ \nabla_2+\nabla_5, \nabla_3+\nabla_4 \rangle;$

\item if $2\alpha_3 \beta_3 \beta_4^2 = \beta_1\beta_4^2 - \beta_2 \beta_3^2,$  $4\alpha_2\beta_4^2 = - \beta_2^2$ and
$\beta_1^2 \neq 4 \alpha_3 \beta_1 \beta_3-4 \alpha_1 \beta_3^2,$
    then choosing 
    $x=\frac{\beta_4\sqrt{\beta_1^2- 4 \alpha_3 \beta_1 \beta_3+4 \alpha_1 \beta_3^2}}{2\beta_3^2},$
    $z=-\frac{x\beta_1}{2\beta_3},$  $u=\frac{x\beta_3}{\beta_4}$ and $v=-\frac{u\beta_2}{2\beta_4},$
     we have the  representative
    $\langle  \nabla_1+\nabla_5, \nabla_3+\nabla_4 \rangle;$

\item if $2\alpha_3 \beta_3 \beta_4^2 = \beta_1\beta_4^2 - \beta_2 \beta_3^2,$  $4\alpha_2\beta_4^2 = - \beta_2^2$ and
$\beta_1^2 = 4 \alpha_3 \beta_1 \beta_3-4 \alpha_1 \beta_3^2,$
    then choosing $x=\beta_4,$ $z=-\frac{\beta_1, \beta_4}{2\beta_3},$  $u=\beta_3$ and $v=-\frac{\beta_2\beta_3}{2\beta_4},$
         we have the representative
    $\langle  \nabla_5, \nabla_3+\nabla_4 \rangle.$

\end{enumerate}

\item   if $\beta_3= 0,$ $\beta_4\neq 0$ then we may suppose $\alpha_4=0,$ and
\begin{enumerate}
    \item if $\beta_1\neq 0$ and $\beta_2^2\neq - 4 \alpha_2 \beta_4^2,$
    then choosing
    $x=\frac{\sqrt{\beta_2^2+4 \alpha_2 \beta_4^2}}{2 \beta_4},$
    $z=-x\alpha_3,$ $u=\sqrt{\frac{x\beta_1} {\beta_4}}$ and $v=-\frac{u\beta_2}{2\beta_4},$
    we have the family of representatives
    $\langle  \alpha \nabla_1+\nabla_2+\nabla_5,  \nabla_1+\nabla_4 \rangle;$

     \item if $\beta_1\neq 0,$  $\beta_2^2= - 4 \alpha_2 \beta_4^2$
     and
     $ \beta_1\beta_2\neq 2\beta^2_4(\alpha^2_3-\alpha_1) ,$
    then choosing
     $x=\frac{\beta_1\beta_2+ 2\beta^2_4(\alpha_1-\alpha_3^2)}{2\beta_1\beta_4},$ 
    $z=-x\alpha_3,$ $u=\sqrt{\frac{x\beta_1} {\beta_4}}$ and $v=-\frac{u\beta_2}{2\beta_4},$
    we have the representative
    $\langle   \nabla_1+ \nabla_5,  \nabla_1+\nabla_4 \rangle;$

     \item if $\beta_1\neq 0,$  $\beta_2^2= - 4 \alpha_2 \beta_4^2$ and
    $ \beta_1\beta_2= 2\beta^2_4(\alpha^2_3-\alpha_1) ,$
    then choosing
    $x=\beta_4,$
    $z=-\alpha_3 \beta_4,$  $u=\sqrt{\beta_1}$
    and
    $v=-\frac{\sqrt{\beta_1} \beta_2}{2 \beta_4},$
    we have the representative
    $\langle   \nabla_5,  \nabla_1+\nabla_4 \rangle;$

    \item if $\beta_1=0,$ $\alpha_1\neq \alpha_3^2$ and $\beta_2^2\neq -4 \alpha_2 \beta_4^2,$ then by choosing
    $x= \frac{\sqrt{\beta_2^2+4 \alpha_2 \beta_4^2}}
    {2\beta_4 },$
    $z=-x\alpha_3,$ 
    $u=\sqrt{\alpha_1-\alpha_3^2}$
     and $v=-\frac{u\beta_2}{2\beta_4},$
    we have the representative
     $\langle  \nabla_1+ \nabla_2+\nabla_5, \nabla_4 \rangle;$

    \item if $\beta_1=0,$ $\alpha_1\neq \alpha_3^2$ and $\beta_2^2= -4 \alpha_2 \beta_4^2,$ then choosing
    $x=1,$
    $z=-\alpha_3,$ $u=\sqrt{\alpha_1-\alpha_3^2}$
    and
    $v=-\frac{u\beta_2}{2\beta_4},$
    we have the representative
    $\langle  \nabla_1+ \nabla_5,  \nabla_4 \rangle;$

    \item if $\beta_1=0,$ $\alpha_1= \alpha_3^2$ and $\beta_2^2\neq -4 \alpha_2 \beta_4^2 $ then choosing
    $x=\frac{\sqrt{\beta_2^2+4 \alpha_2 \beta_4^2}}{2 \beta_4},$
    $u=2 \beta_4,$
    $z=-x \alpha_3$
    and
    $v=-\beta_2,$
    we have the representative
      $\langle \nabla_2+ \nabla_5,  \nabla_4 \rangle;$

     \item if $\beta_1=0,$ $\alpha_1= \alpha_3^2$ and $\beta_2^2= -4 \alpha_2 \beta_4^2 $ then choosing
        $x=1,$
        $u=2 \beta_4,$
        $z=-\alpha_3$
        and
        $v=-\beta_2,$
        we have the representative
        $\langle \nabla_5,  \nabla_4 \rangle.$

\end{enumerate}

\item   if $\beta_3\neq 0$ and  $\beta_4= 0,$
then by applying $\phi_2$ with $x=u=1$ and $z=v=0,$ we have $\beta_3= 0$ and  $\beta_4\neq0.$
It is the case considered above.

\item if $\beta_3=0$ and $\beta_4=0,$ then

\begin{enumerate}
    \item if $\beta_1\neq 0,$ $\beta_2\neq 0,$ $\alpha_2=0$ and $\alpha_4^2 \beta_1\neq -(\alpha_1-\alpha_3^2) \beta_2,$
    then choosing
    $x=\sqrt{\frac{\alpha_4^2 \beta_1+(\alpha_1-\alpha_3^2) \beta_2}{\beta_1}},$
    $z=-x \alpha_3,$ $u=x \sqrt{\frac{\beta_1}{\beta_2}},$
    and
    $v=-u \alpha_4,$
    we have the representative
      $\langle \nabla_1+\nabla_5,  \nabla_1+\nabla_2 \rangle;$

   \item  if $\beta_1\neq 0,$ $\beta_2\neq 0,$ $\alpha_2=0$ and $\alpha_4^2 \beta_1= -(\alpha_1-\alpha_3^2) \beta_2,$
    then choosing
    $x=\sqrt{\beta_2},$
    $u=\sqrt{\beta_1},$
    $z=-\alpha_3 \sqrt{\beta_2}$
    and
    $v=-\alpha_4 \sqrt{\beta_1},$
    we have the representative
      $\langle \nabla_5,  \nabla_1+\nabla_2 \rangle;$

    \item if $\beta_1= 0,$ $\beta_2\neq 0,$ $\alpha_2=0$ and $\alpha_1\neq \alpha_3^2,$
    then choosing
    $x=1,$
    $u=\sqrt{\alpha_1-\alpha_3^2},$
    $z=-\alpha_3$
    and
    $v=u \alpha_4,$
    we have the representative
      $\langle \nabla_1+\nabla_5,  \nabla_2 \rangle;$

   \item if $\beta_1= 0,$ $\beta_2\neq 0,$ $\alpha_2=0$ and $\alpha_1= \alpha_3^2,$
    then choosing
    $x=1,$
    $u=1,$
    $z=-\alpha_3$
    and
    $v=-\alpha_4 ,$
    we have the representative
      $\langle  \nabla_5,  \nabla_2 \rangle;$

 \item if $\beta_1\neq 0$ and $\beta_2= 0,$ then by applying $\phi_2$ with $x=u=1$ and $z=v=0,$ we have $\beta_1= 0$ and  $\beta_2\neq0.$ It is the case considered above.

\end{enumerate}

\end{enumerate}

Summarizing, all cases, we have the following distinct  orbits 

\begin{longtable}{lll}
$\langle \alpha \nabla_1+\beta \nabla_2+\nabla_3+\nabla_5, \nabla_3+\nabla_4 \rangle$ &
$\langle \alpha \nabla_1+ \nabla_2+\nabla_5, \nabla_3+\nabla_4 \rangle$&
$\langle  \nabla_1+\nabla_5, \nabla_3+\nabla_4 \rangle$\\
$\langle  \nabla_5, \nabla_3+\nabla_4 \rangle$ &
 $\langle  \alpha \nabla_1+\nabla_2+\nabla_5,  \nabla_1+\nabla_4 \rangle$ &
 $\langle   \nabla_1+ \nabla_5,  \nabla_1+\nabla_4 \rangle$ \\
  $\langle   \nabla_5,  \nabla_1+\nabla_4 \rangle$ &
  $\langle  \nabla_1+ \nabla_2+\nabla_5, \nabla_4 \rangle$ &
   $\langle  \nabla_1+ \nabla_5, \nabla_4 \rangle$ \\
    $\langle  \nabla_2+ \nabla_5, \nabla_4 \rangle$&
     $\langle \nabla_5, \nabla_4 \rangle$ &
     $\langle \nabla_1+\nabla_5,  \nabla_1+\nabla_2 \rangle$\\
      $\langle \nabla_5,  \nabla_1+\nabla_2 \rangle$ &
      $\langle \nabla_1+\nabla_5,  \nabla_2 \rangle$ &
       $\langle  \nabla_5,  \nabla_2 \rangle$
 \end{longtable}

Hence, we have the following new $5$-dimensional nilpotent $\mathfrak{CCD}$-algebras:

\begin{longtable}{llllllllll}
$\mathfrak{C}^5_{26}(\alpha,\beta)$ & $: $ & $e_1e_1=\alpha e_5$ & $e_1e_2=e_3$ & $e_2e_2=\beta e_5$ & $e_1e_3=e_4+e_5$ & $e_2e_3=e_4$ & $e_3e_3=e_5 $\\
$\mathfrak{C}^5_{27}(\alpha)$ & $ : $ & $e_1e_1=\alpha e_5$ & $e_1e_2=e_3$ & $e_2e_2=e_5$ & $e_1e_3=e_4$ & $e_2e_3=e_4$ & $e_3e_3=e_5 $\\
$\mathfrak{C}^5_{28}$ & $ : $ & $e_1e_1=e_5$ & $e_1e_2=e_3$ & $e_1e_3=e_4$ & $e_2e_3=e_4$ & $e_3e_3=e_5 $\\
$\mathfrak{C}^5_{29}$ & $ : $ & $e_1e_2=e_3$ & $e_1e_3=e_4$ & $e_2e_3=e_4$ & $e_3e_3=e_5 $\\
$\mathfrak{C}^5_{30}(\alpha)$ & $ : $ & $e_1e_1=e_4+\alpha e_5$ & $e_1e_2=e_3$ & $e_2e_2=e_5$ & $e_2e_3=e_4$ & $e_3e_3=e_5 $\\
$\mathfrak{C}^5_{31}$ & $ : $ & $e_1e_1=e_4+e_5$ & $e_1e_2=e_3$ & $e_2e_3=e_4$ & $e_3e_3=e_5 $\\
$\mathfrak{C}^5_{32}$ & $ : $ & $e_1e_1=e_4$ & $e_1e_2=e_3$ & $e_2e_3=e_4$ & $e_3e_3=e_5 $\\
$\mathfrak{C}^5_{33}$ & $ : $ & $e_1e_1=e_5$ & $e_1e_2=e_3$ & $e_2e_2=e_5$ & $e_2e_3=e_4$ & $e_3e_3=e_5 $\\
$\mathfrak{C}^5_{34}$ & $ : $ & $e_1e_1=e_5$ & $e_1e_2=e_3$ & $e_2e_3=e_4$ & $e_3e_3=e_5 $\\
$\mathfrak{C}^5_{35}$ & $ :$ & $e_1e_2=e_3$ & $e_2e_2=e_5$ & $e_2e_3=e_4$ & $e_3e_3=e_5 $\\
$\mathfrak{C}^5_{36}$ & $ :$ & $e_1e_2=e_3$ & $e_2e_3=e_4$ & $e_3e_3=e_5 $\\
$\mathfrak{C}^5_{37}$ & $ :$ & $e_1e_1=e_4+e_5 $ & $e_1e_2=e_3$ & $e_2e_2=e_4$ & $e_3e_3=e_5 $\\
$\mathfrak{C}^5_{38}$ & $ :$ & $e_1e_1=e_4$ & $e_1e_2=e_3$ & $e_2e_2=e_4$ & $e_3e_3=e_5 $\\
$\mathfrak{C}^5_{39}$ & $ :$ & $e_1e_1=e_5$ & $e_1e_2=e_3$ & $e_2e_2=e_4$ & $e_3e_3=e_5 $\\
$\mathfrak{C}^5_{40}$ & $ : $ & $e_1e_2=e_3$ & $e_2e_2=e_4$ & $e_3e_3=e_5 $
\end{longtable}
 where 
$\mathfrak{C}^5_{26}(\alpha,\beta) \cong  \mathfrak{C}^5_{26}(\beta, \alpha)$ 
 and $\mathfrak{C}^5_{27}(\alpha) \cong  \mathfrak{C}^5_{27}(\frac 1 {\alpha}).$ 
 
 \subsection{2-dimensional central extensions of ${\mathfrak{C}_{01}^{3}}$}
Here we will  collect all information about ${\mathfrak{C}_{01}^{3}}:$

\begin{longtable}{|l|ll|l|}
\hline
Algebra  & \multicolumn{2}{l|}{Multiplication} & Cohomology   \\
\hline
${\mathfrak{C}}_{01}^{3}$ &  $   e_1 e_1 = e_2$&$ e_2 e_2 = e_3$ & 
$\begin{array}{lcl}
{\rm H^2_{\mathfrak{J}}}({\mathfrak{C}}_{01}^{3})&=& \varnothing  \\
{\rm H^2_{\mathfrak{CCD}}}({\mathfrak{C}}_{01}^{3})&=& \langle [\Delta_{12}]\rangle 
\end{array}$\\
\hline
\end{longtable}
It is follow that there are no non split $\mathfrak{CCD}$-central extensions of ${\mathfrak{C}}_{01}^{3}.$

\medskip

\section{Central extensions of   $4$-dimensional nilpotent commutative $\mathfrak{CD}$ algebras}
\subsection{$1$-dimensional central extensions of $\mathfrak{C}^{4*}_{01}$.}
Here we will  collect all information about ${\mathfrak{C}_{01}^{4*}}:$

\begin{longtable}{|l|l|l|l|}
\hline
Algebra  & Multiplication & Cohomology &  Automorphisms  \\
\hline
${\mathfrak{C}}_{01}^{4*}$ &  $  e_1 e_1 = e_2 $ & 
$\begin{array}{lcl}
{\rm H^2_{\mathfrak{J}}}({\mathfrak{C}}_{01}^{4*})&=&\langle [\Delta_{ij}]  \rangle_{1\leq i \leq j\leq 4}^{ (i,j)\not\in \{  (1,1), \ (2,2) \} } \\
{\rm H^2_{\mathfrak{CCD}}}({\mathfrak{C}}_{01}^{4*})&=&{\rm H^2_{\mathfrak{J}}}({\mathfrak{C}_{01}^{4*}}) \oplus \langle[\Delta_{22}]\rangle 
\end{array}$& 
$\phi=
	\begin{pmatrix}
	x &  0  & 0 & 0\\
	q &  x^2& r & u\\
	w &  0  & t & k\\
	e &  0  & y & l
	\end{pmatrix}
	$\\
\hline
\end{longtable}

Let us use the following notations:
\begin{longtable}{lllllll}
$\nabla_1 = [\Delta_{12}]$ & $\nabla_2 = [\Delta_{13}]$ & $\nabla_3 = [\Delta_{14}]$ \\
$\nabla_4 = [\Delta_{23}]$ & $\nabla_5 = [\Delta_{24}]$ & $\nabla_6 = [\Delta_{33}]$ \\ 
$\nabla_7 = [\Delta_{34}]$ & $\nabla_8 = [\Delta_{44}]$ & $\nabla_9 =[\Delta_{22}]$
\end{longtable}

Take $\theta=\sum\limits_{i=1}^9\alpha_{i}\nabla_{i}\in {\rm H}_{\mathfrak{CCD}}^2(\mathfrak{C}^{4*}_{01}).$
We are interested only in $\alpha_9\neq 0$,  $(\alpha_2, \alpha_4,\alpha_6,\alpha_7)\neq (0,0,0,0),$ and $(\alpha_3, \alpha_5,\alpha_7,\alpha_8)\neq (0,0,0,0).$ Then we can suppose that $\alpha_9=1.$ 
Since
$$
\phi^T\begin{pmatrix}
0 & \alpha_1 & \alpha_2 & \alpha_3 \\
\alpha_1 & 1 & \alpha_4 & \alpha_5 \\
\alpha_2 & \alpha_4 & \alpha_6  & \alpha_7\\
\alpha_3 & \alpha_5 & \alpha_7  & \alpha_8
\end{pmatrix} \phi=
\begin{pmatrix}
\alpha^*   & \alpha^*_1 & \alpha^*_2 & \alpha^*_3 \\
\alpha^*_1 & \alpha^*_9 & \alpha^*_4 & \alpha^*_5 \\
\alpha^*_2 & \alpha^*_4 & \alpha^*_6 & \alpha^*_7\\
\alpha^*_3 & \alpha^*_5 & \alpha^*_7 & \alpha^*_8
\end{pmatrix},
$$ 	
we have
	\begin{longtable}{lcl}
	$\alpha^*_1$ &$=$& $x^2 (q+x \alpha_1+w \alpha_4+e \alpha_5)$\\
	$\alpha^*_2$&$=$& $r (q+x \alpha_1+w \alpha_4+e \alpha_5)+t (x \alpha_2+q \alpha_4+w \alpha_6+e \alpha_7)+y (x \alpha_3+q \alpha_5+w \alpha_7+e \alpha_8)$\\
	$\alpha^*_3$&$=$& $u (q+x \alpha_1+w \alpha_4+e \alpha_5)+k (x \alpha_2+q \alpha_4+w \alpha_6+e \alpha_7)+l (x \alpha_3+q \alpha_5+w \alpha_7+e \alpha_8)$\\
	$\alpha^*_4$&$=$& $x^2 (r+t \alpha_4+y \alpha_5)$\\
	$\alpha^*_5$&$=$& $x^2 (u+k \alpha_4+l \alpha_5)$\\
    $\alpha^*_6$&$=$& $r^2+2 r (t \alpha_4+y \alpha_5)+t^2 \alpha_6+2 t y \alpha_7+y^2 \alpha_8$\\
    $\alpha^*_7$&$=$& $r (u+k \alpha_4+l \alpha_5)+t (u \alpha_4+k \alpha_6+l \alpha_7)+y (u \alpha_5+k \alpha_7+l \alpha_8)$\\
    $\alpha^*_8$&$=$& $u^2+2 u (k \alpha_4+l \alpha_5)+k^2 \alpha_6+2 k l \alpha_7+l^2 \alpha_8$\\
    $\alpha^*_9$&$=$& $x^4$\\	\end{longtable}

By choosing 
$q = -x \alpha_1 - w \alpha_4 - e \alpha_5,$
$r = -t \alpha_4 - y \alpha_5$
and $u = -k \alpha_4 - l \alpha_5$, we have $\alpha_1^*=\alpha_4^*=\alpha_5^*=0.$
Now we can suppose that  $\alpha_1=\alpha_4=\alpha_5=0$.
Then:

\begin{enumerate}
    \item if $\alpha_8\neq 0$ and $ \alpha_7^2\neq \alpha_6 \alpha_8,$ by choosing 
    $q=0,$
    $w=- \alpha_3 \alpha_7+ \alpha_2 \alpha_8,$
    $e= \alpha_3 \alpha_6- \alpha_2 \alpha_7,$
    $r=0,$
    $t=1,$
    $y=\frac{- \alpha_7+\sqrt{ \alpha_7^2- \alpha_6 \alpha_8}}{ \alpha_8},$
    $u=0,$
    $k=\frac{ \alpha_8 ( \alpha_6 \alpha_8- \alpha_7^2)^3}{2},$
    $l=\frac{ \alpha_7^7 \alpha_8-3 \alpha_6 \alpha_7^5 \alpha_8^2+3 \alpha_6^2 \alpha_7^3 \alpha_8^3- \alpha_6^3 \alpha_7 \alpha_8^4+\sqrt{ \alpha_8^2 ( \alpha_7^2- \alpha_6 \alpha_8)^7}}{2 \alpha_8}$
    and 
    $x= \alpha_7^2- \alpha_6 \alpha_8,$ we have the representative $\langle \nabla_7+\nabla_9 \rangle;$
    
    \item if $\alpha_8\neq 0, $  $ \alpha_7^2= \alpha_6 \alpha_8$ and 
    $\alpha_3 \alpha_7\neq \alpha_2 \alpha_8,$
    by choosing 
    $q=0,$ 
    $w=0,$
    $e=-\frac{\alpha_3}{\alpha_8},$
    $r=0,$
    $t=\frac{\alpha_8}{\alpha_2 \alpha_8-\alpha_3 \alpha_7},$
    $y=\frac{\alpha_7}{\alpha_3 \alpha_7-\alpha_2 \alpha_8},$
    $u=0,$
    $k=0,$
    $l=\frac{1}{\sqrt{\alpha_8}}$
    and 
    $x=1,$ we have the representative
    $\langle \nabla_2+\nabla_8+\nabla_9 \rangle;$

    \item if $\alpha_8\neq 0, $  $ \alpha_7^2= \alpha_6 \alpha_8$ and 
    $\alpha_3 \alpha_7= \alpha_2 \alpha_8,$
    by choosing 
    $q=0,$ 
    $w=0,$
    $e=-\alpha_3,$
    $r=0,$
    $t=\alpha_8,$
    $y=-\alpha_7   ,$
    $u=0,$
    $k=0,$
    $l=\sqrt{\alpha_8^3}$
    and 
    $x=\alpha_8,$ we have the representative
    $\langle \nabla_8+\nabla_9 \rangle,$ which gives a split extension;
    
    \item if $\alpha_8=0,$ then by choosing some suitable automorphism, we can take an element with $\alpha_8^*\neq0.$ It is a case considered above. 
    
\end{enumerate}

Summarizing, all cases,  we have only two distinct orbits:
\begin{longtable}
{ccc}
$\langle \nabla_7+\nabla_9 \rangle$ & \ & 
$\langle \nabla_2+\nabla_8+\nabla_9 \rangle$ 
\end{longtable}

Hence we have the following $5$-dimensional algebras:
\begin{longtable}{llllllllll}
$\mathfrak{C}_{41}^5$& $ : $ & $e_1e_1=e_2$ & $e_2e_2=e_5$ &$e_3e_4=e_5$\\
$\mathfrak{C}_{42}^5$& $ : $ & $e_1e_1=e_2$ & $e_1e_3=e_5$ &$e_2e_2=e_5$ &$e_4e_4= e_5$
\end{longtable}

 \subsection{$1$-dimensional central extensions of ${\mathfrak{C}^{4*}_{02}}$.}
Here we will  collect all information about ${\mathfrak{C}_{02}^{4*}}:$

\begin{longtable}{|l|l|l|l|}
\hline
Algebra  & Multiplication & Cohomology &  Automorphisms  \\
\hline
${\mathfrak{C}}_{02}^{4*}$ &  
$\begin{array}{l}e_1 e_1 = e_2 \\ e_1 e_2=e_3 \end{array}$
& 
$\begin{array}{lcl}
{\rm H^2_{\mathfrak{J}}}({\mathfrak{C}}_{02}^{4*})&=&
\left\langle 
[\Delta_{13}+\Delta_{22}],
[\Delta_{14}],[\Delta_{24}],[\Delta_{44}]
\right\rangle \\
{\rm H^2_{\mathfrak{CCD}}}({\mathfrak{C}}_{02}^{4*})&=&
{\rm H^2_{\mathfrak{J}}}({\mathfrak{C}_{02}^{4*}}) \oplus \langle[\Delta_{13}]\rangle 
\end{array}$& 
$\phi=
	\begin{pmatrix}
	x &  0  &  0 & 0\\
	q &  x^2  &  0 & 0\\
	w &  2xq  & x^3 &  r\\
	e &  0 & 0 & t
	\end{pmatrix} 
	$\\
\hline
\end{longtable}

Let us use the following notations:
\begin{longtable}{ccccc}
	$\nabla_1 = [\Delta_{13}+\Delta_{22}]$&$ \nabla_2 = [\Delta_{14}] $&$ \nabla_3 = [\Delta_{24}] $&$\nabla_4 =[\Delta_{44}]$&$	\nabla_5 = [\Delta_{13}]$
	\end{longtable}

Take $\theta=\sum\limits_{i=1}^5\alpha_{i}\nabla_{i}\in {\rm H}_{\mathfrak{CCD}}^2(\mathfrak{C}^{4*}_{02}).$
We are interested only in $\alpha_5\neq 0,$ $\alpha_1\neq -\alpha_5$ and 
$(\alpha_2, \alpha_3,\alpha_4)\neq (0,0,0),$ and can suppose that $\alpha_5=1$ and $\alpha_1\neq -1.$ 
	Since
	$$
	\phi^T\begin{pmatrix}
	0 & 0 & \alpha_1+1 & \alpha_2 \\
	0  &  \alpha_1 & 0 & \alpha_3 \\
	\alpha_1+1 & 0 & 0  & 0\\
	\alpha_2 & \alpha_3 & 0  & \alpha_4
	\end{pmatrix} \phi=
	\begin{pmatrix}
	\alpha^* & \alpha^{**} & \alpha^*_1+\alpha_5^* & \alpha^*_2 \\
	\alpha^{**}  &  \alpha^*_1 & 0 & \alpha^*_3 \\
	\alpha^*_1+\alpha_5^* & 0 & 0   & 0\\
	\alpha^*_2 & \alpha^*_3& 0 & \alpha^*_4
	\end{pmatrix},
	$$ 
we have
	\begin{longtable}{lcllcllcl}
$\alpha^*_1$ &$=$&$x^4 \alpha_1$ &
$\alpha^*_2$ &$=$&$r x (1+\alpha_1)+t (x \alpha_2+q \alpha_3+e \alpha_4)$ &
$\alpha^*_3$ &$=$&$t x^2 \alpha_3$ \\
$\alpha^*_4$ &$=$&$t^2 \alpha_4$&
$\alpha^*_5$ &$=$&$x^4$ 
	\end{longtable}
	
Hence, we have the following cases.
\begin{enumerate}
\item If $\alpha_4\neq 0,$
then by choosing 
$q=0,$
$w=0,$
$e=-\alpha_2,$
$r=0,$
$t=\sqrt{\alpha_4^3}$
and 
$x=\alpha_4,$
we have the following family of representatives
$\langle \alpha \nabla_1+\beta \nabla_3+\nabla_4+\nabla_5\rangle.$
If $(\alpha,\beta)\ne(\alpha',\beta')$, then 
\[\orb\langle \alpha \nabla_1+\beta \nabla_3+\nabla_4+\nabla_5\rangle
=\orb\langle \alpha' \nabla_1+\beta' \nabla_3+\nabla_4+\nabla_5\rangle
\] if and only if  $(\alpha',\beta')=(\alpha,-\beta).$

\item If $\alpha_4= 0$ and $\alpha_3\neq0,$
then by choosing 
$q=-\alpha_2,$
$w=0,$
$e=0,$
$r=0,$
$t=\alpha_3$
and 
$x=\alpha_3,$
we have the following family of distinct representatives
$\langle \alpha \nabla_1+\nabla_3+ \nabla_5\rangle.$

\item If $\alpha_4= 0$, $\alpha_3=0$,
then by choosing $r=-\frac{t\alpha_2}{\alpha_1+1},$  we have $(\alpha_2^*,\alpha_3^*,\alpha_4^*)=(0,0,0)$ and it is a split extension.

\end{enumerate}

Summarizing, all cases, we have the following distinct orbits:
\begin{longtable}{lllll}
$\langle \alpha \nabla_1+\nabla_3+ \nabla_5\rangle$ & \ & 
$\langle \alpha \nabla_1+\beta \nabla_3+\nabla_4+\nabla_5\rangle$  
\end{longtable}
	
Hence we have the following $5$-dimensional algebras:
\begin{longtable}{llllllllll}
$\mathfrak{C}_{12}^5(\alpha)$&$:$& 
$e_1e_1=e_2$ & $e_1e_2=e_3$ &$e_1e_3=(\alpha+1) e_5$ &$e_2e_2= \alpha e_5$  &$e_2e_4= e_5$\\
$\mathfrak{C}_{13}^5(\alpha, \beta)$&$:$& 
$e_1e_1=e_2$ & $e_1e_2=e_3$ &$e_1e_3=(\alpha+1) e_5$ &$e_2e_2= \alpha e_5$ &$e_2e_4= \beta e_5$ &$e_4e_4= e_5$\end{longtable}
where
$\mathfrak{C}_{13}^5(\alpha, \beta) \cong         
 \mathfrak{C}_{13}^5(\alpha,-\beta).$

 \subsection{$1$-dimensional central extensions of ${\mathfrak{C}^{4*}_{03}}$.}
Here we will  collect all information about ${\mathfrak{C}_{03}^{4*}}:$

\begin{longtable}{|l|l|l|}
\hline
Algebra  & Multiplication & Cohomology   \\
\hline
${\mathfrak{C}}_{03}^{4*}$ &  $  e_1 e_2 = e_3 $ & 
$\begin{array}{lcl}
{\rm H^2_{\mathfrak{J}}}({\mathfrak{C}}_{03}^{4*})&=&\langle [\Delta_{ij}]  \rangle_{1\leq i \leq j \leq 4}^{ (i,j) \not \in \{ (1,2), \ (3,3) \} } \\
{\rm H^2_{\mathfrak{CCD}}}({\mathfrak{C}}_{03}^{4*})&=&
{\rm H^2_{\mathfrak{J}}}({\mathfrak{C}_{03}^{4*}}) \oplus \langle[\Delta_{33}]\rangle 
\end{array}$\\
\hline
\multicolumn{3}{|l|}{ Automorphisms}\\
\hline	
\multicolumn{3}{|c|}{  
$
	\phi_1=
	\begin{pmatrix}
	x &  0  &  0 & 0\\
	0 &  e  &  0 & 0\\
	q &  r  & x e & y\\
	w &  t  & 0 &  u
	\end{pmatrix} \quad
	\phi_2=
	\begin{pmatrix}
	0 &  x  &  0 & 0\\
	e &  0  &  0 & 0\\
	q &  r  &  x e & y\\
	w &  t  &  0 &  u
	\end{pmatrix}
	$}\\
	
\hline
\end{longtable}
Let us use the following notations:
\begin{longtable}{ccc}
$\nabla_1 = [\Delta_{11}]$ & $\nabla_2 = [\Delta_{13}]$ & $\nabla_3 = [\Delta_{14}]$\\ $\nabla_4 = [\Delta_{22}]$ & $\nabla_5 = [\Delta_{23}]$ & $\nabla_6 = [\Delta_{24}]$\\ $\nabla_7 = [\Delta_{34}]$ & $\nabla_8 = [\Delta_{44}]$ & $\nabla_9 = [\Delta_{33}]$
	\end{longtable}
	
Take $\theta=\sum\limits_{i=1}^9\alpha_{i}\nabla_{i}\in {\rm H}_{\mathfrak{CCD}}^2(\mathfrak{C}^{4*}_{03}).$
We are interested only in $\alpha_9\neq 0$ and  
$(\alpha_3, \alpha_6,\alpha_7,\alpha_8)\neq (0,0,0,0),$ then can suppose $\alpha_9=1.$ 
	Since
	$$
	\phi_1^T\begin{pmatrix}
	\alpha_1 & 0 & \alpha_2 & \alpha_3 \\
	0  &  \alpha_4 & \alpha_5 & \alpha_6 \\
	\alpha_2 & \alpha_5 & 1   & \alpha_7\\
	\alpha_3 & \alpha_6 & \alpha_7  & \alpha_8
	\end{pmatrix} \phi_1=
	\begin{pmatrix}
	\alpha^*_1 & \alpha^* & \alpha^*_2 & \alpha^*_3 \\
	\alpha^*  &  \alpha^{*}_4 & \alpha^*_5 & \alpha^*_6 \\
	\alpha^*_2 & \alpha^*_5 & \alpha^*_9   & \alpha^*_7\\
	\alpha^*_3 & \alpha^*_6 & \alpha^*_7  & \alpha^*_8
	\end{pmatrix},
	$$ 
we have
	\begin{longtable}{lcl}
$\alpha_1^*$ & $=$ & $q^2+x^2 \alpha_1+2 w x \alpha_3+2 q (x \alpha_2+w \alpha_7)+w^2 \alpha_8$\\
$\alpha^*_2$ & $=$ & $e x (q+x \alpha_2+w \alpha_7)$\\
$\alpha^*_3$ & $=$ & $x y \alpha_2+u x \alpha_3+w y \alpha_7+q (y+u \alpha_7)+u w \alpha_8$\\
$\alpha^*_4$ & $=$ & $r^2+e^2 \alpha_4+2 e t \alpha_6+2 r (e \alpha_5+t \alpha_7)+t^2 \alpha_8$\\
$\alpha^*_5$ & $=$ & $e x (r+e \alpha_5+t \alpha_7)$\\
$\alpha^*_6$ & $=$ & $e y \alpha_5+e u \alpha_6+t y \alpha_7+r (y+u \alpha_7)+t u \alpha_8$\\
$\alpha^*_7$ & $=$ & $e x (y+u \alpha_7)$\\
$\alpha^*_8$ & $=$ & $y^2+2 u y \alpha_7+u^2 \alpha_8$\\
$\alpha^*_9$ & $=$ & $e^2 x^2$
	\end{longtable}

By choosing 
$q = -x \alpha_2 - w \alpha_7,$
$r = -e \alpha_5 - t \alpha_7$
and
$y = -u \alpha_7,$ we have $\alpha_2^*=\alpha_5^*=\alpha_7^*=0.$
Now we can suppose that  $\alpha_2=\alpha_5=\alpha_7=0$ and $\alpha_9=1.$
Then:
\begin{enumerate}
    \item if $\alpha_8\neq 0,$  $\alpha_3^2\neq \alpha_1 \alpha_8$ and $\alpha_6^2\neq\alpha_4 \alpha_8,$
    by choosing 
    $q=0,$
    $w=-\frac{\alpha_3}{\alpha_8} \sqrt{\frac{\alpha_4 \alpha_8-\alpha_6^2}{\alpha_8}},$
    $e=\sqrt{\frac{\alpha_1 \alpha_8-\alpha_3^2}{\alpha_8}},$
    $r=0,$
    $t= - \frac{\alpha_6}{\alpha_8}  \sqrt{\frac{\alpha_1 \alpha_8-\alpha_3^2}{\alpha_8}},$
    $u=\sqrt{\frac{(\af_1 \af_8-\af_3^2)(\af_4 \af_8-\af_6^2)}{ \af_8^3}}$
    and
    $x=\sqrt{\frac{\alpha_4 \alpha_8-\alpha_6^2}{\alpha_8}},$
    we have the representative 
    $\langle \nabla_1+\nabla_4+\nabla_8+\nabla_9\rangle;$

  \item if $\alpha_8\neq 0,$  $\alpha_3^2\neq \alpha_1 \alpha_8$ and $\alpha_6^2=\alpha_4 \alpha_8,$
    by choosing 
    $q=0,$
    $w=-\alpha_3,$
    $e=\sqrt{\frac{\alpha_1 \alpha_8-\alpha_3^2}{\alpha_8}},$
    $r=0,$
    $t=  -\frac{\alpha_6}{\alpha_8}  \sqrt{\frac{\alpha_1 \alpha_8-\alpha_3^2}{\alpha_8}},$
    $u=\sqrt{\alpha_1 \alpha_8-\alpha_3^2}$
    and
    $x= \alpha_8,$
    we have the representative 
    $\langle \nabla_1 +\nabla_8+\nabla_9\rangle;$

  \item if $\alpha_8\neq 0,$  $\alpha_3^2= \alpha_1 \alpha_8$ and $\alpha_6^2 \neq \alpha_4 \alpha_8,$
    by choosing 
    $q=0,$
    $w=-\frac{\alpha_3}{\alpha_8} \sqrt{\frac{\af_4\af_8-\af_6^2}{\af_8}},$
    $e= {\alpha_8},$
    $r=0,$
    $t=-\alpha_6,$
    $u=\sqrt{\af_4\af_8-\af_6^2}$
    and
    $x= \sqrt{\frac{\af_4\af_8-\af_6^2}{\af_8}},$
    we have the representative 
    $\langle \nabla_4 +\nabla_8+\nabla_9\rangle,$
    which after action of some suitable automorphism $\phi_2$ gives     $\langle \nabla_1 +\nabla_8+\nabla_9\rangle;$

  \item if $\alpha_8\neq 0,$  $\alpha_3^2= \alpha_1 \alpha_8$ and $\alpha_6^2 = \alpha_4 \alpha_8,$
    by choosing 
    $q=0,$
    $w=- {\alpha_3},$
    $e= {\alpha_8},$
    $r=0,$
    $t=-\alpha_6,$
    $u=\sqrt{\alpha_8^3}$
    and
    $x=  {\af_8},$
    we have the representative 
    $\langle \nabla_8+\nabla_9\rangle;$

\item if $\alpha_8= 0,$  $\alpha_6\neq  0$ and $\af_3\neq0,$ by choosing 
    $q=0,$
    $w=-\af_1 \af_6^2,$
    $e=2 \af_3^2 \af_6,$
    $r=0,$
    $t=-\af_3^2 \af_4,$
    $y=0,$
    $u=8 \af_3^4 \af_6^4$
    and
    $x=2 \af_3 \af_6^2,$
    we have the representative     $\langle \nabla_3+\nabla_6+  \nabla_9\rangle;$

\item if $\alpha_8= 0,$  $\alpha_6\neq  0,$ $\af_3= 0$ and $\af_1\neq 0,$ by choosing 
    $q=0,$
    $w=0,$
    $e=\sqrt{\alpha_1},$
    $r=0,$
    $t=-\frac{\sqrt{\alpha_1} \af_4}{2\alpha_6},$
    $y=0,$
    $u=\frac{\sqrt{\alpha_1}}{\af_6}$ and 
    $x=1,$ we have the representative
    $\langle \nabla_1+\nabla_6+  \nabla_9\rangle$;

\item if $\alpha_8= 0,$  $\alpha_6\neq  0,$ $\af_3\neq 0$ and $\af_1= 0,$ by choosing 
    $q=0,$
    $w=0,$
    $e= \af_6,$
    $r=0,$
    $t=-\frac{\af_4}{2},$
    $y=0,$ $u=1$ and 
    $x=1,$ we have the representative $\langle \nabla_6+  \nabla_9\rangle$;
    
\item if $\alpha_8= 0$ and  $\alpha_6= 0,$ then $\alpha_3\neq 0$ (in the oposite case, we have a split extension). Then  by choosing some suitable automorphism $\phi_2$ we have $\alpha_6^*\neq0$ and $\alpha_8^*=0.$ It gives a case considered above.

\end{enumerate}

Summarizing, all cases, we have the following distinct orbits:

\begin{longtable}{llll}
$\langle \nabla_1+\nabla_6+  \nabla_9\rangle$ &
$\langle \nabla_1+\nabla_4+\nabla_8+\nabla_9\rangle$&
$\langle \nabla_1+ \nabla_8+\nabla_9\rangle$ \\
$\langle \nabla_3+\nabla_6+  \nabla_9\rangle$ &
$\langle \nabla_6+  \nabla_9\rangle$ &
$\langle \nabla_8+ \nabla_9\rangle$
\end{longtable}

 Hence we have the following $5$-dimensional algebras:
\begin{longtable}{llllllllll}
$\mathfrak{C}_{43}^5$& $ : $ & $e_1e_2=e_3$ & $e_1e_1=e_5$ & $e_2e_4=e_5$ & $e_3e_3=e_5$\\
$\mathfrak{C}_{44}^5$& $ : $ & $e_1e_2=e_3$ & $e_1e_1=e_5$ & $e_2e_2=e_5$ & $e_3e_3=e_5$& $e_4e_4=e_5$\\
$\mathfrak{C}_{45}^5$& $ : $ & $e_1e_2=e_3$ & $e_1e_1=e_5$ & $e_3e_3=e_5$ & $e_4e_4=e_5$\\
$\mathfrak{C}_{46}^5$& $ : $ & $e_1e_2=e_3$ & $e_1e_4=e_5$ & $e_2e_4=e_5$ & $e_3e_3=e_5$\\
$\mathfrak{C}_{47}^5$& $ : $ & $e_1e_2=e_3$ & $e_2e_4=e_5$ & $e_3e_3=e_5$ \\
$\mathfrak{C}_{48}^5$& $ : $ & $e_1e_2=e_3$ & $e_3e_3=e_5$ & $e_4e_4=e_5$\\
\end{longtable}

 \subsection{$1$-dimensional central extensions of ${\mathfrak{C}_{04}^{4*}}$}  
 Here we will  collect all information about ${\mathfrak{C}_{04}^{4*}}:$

\begin{longtable}{|l|l|l|}
\hline
Algebra  & Multiplication & Cohomology    \\
\hline
${\mathfrak{C}}_{04}^{4*}$ &  
$\begin{array}{l}e_1 e_1 = e_3 \\ e_2 e_2=e_4 \end{array}$
& 
$\begin{array}{lcl}
{\rm H^2_{\mathfrak{J}}}({\mathfrak{C}}_{04}^{4*})&=&
\langle [\Delta_{12}],[\Delta_{13}],[\Delta_{14}],[\Delta_{23}],[\Delta_{24}] \rangle \\
 
{\rm H^2_{\mathfrak{CCD}}}({\mathfrak{C}}_{04}^{4*})&=&{\rm H^2_{\mathfrak{J}}}({\mathfrak{C}_{04}^{4*}}) \oplus 
\langle[\Delta_{33}],[\Delta_{34}],[\Delta_{44}]\rangle  
\end{array}$\\
\hline
\multicolumn{3}{|l|}{ Automorphisms}\\
\hline	
\multicolumn{3}{|c|}{  
	$\phi_1=
	\begin{pmatrix}
	x &  0  &  0 & 0\\
	0 &  e  &  0 & 0\\
	q &  r  & x^2 &  0\\
	w &  t  & 0 & e^2
	\end{pmatrix}  \quad 
	\phi_2=
	\begin{pmatrix}
	0 &  x  &  0 & 0\\
	e &  0  &  0 & 0\\
	q &  r  & 0 &  x^2\\
	w &  t  & e^2 & 0
	\end{pmatrix}$}\\
	
\hline
 
\end{longtable}

Let us use the following notations:
\begin{longtable}{llll}
$\nabla_1 = [\Delta_{12}]$ & $\nabla_2 = [\Delta_{13}]$ &
$\nabla_3 = [\Delta_{14}]$ & $\nabla_4 =[\Delta_{23}]$\\
$\nabla_5 = [\Delta_{24}]$ & $\nabla_6 =[\Delta_{33}]$&
$\nabla_7 = [\Delta_{34}]$ & $\nabla_8 =[\Delta_{44}]$
\end{longtable}
Take $\theta=\sum\limits_{i=1}^8\alpha_{i}\nabla_{i}\in {\rm H_{\mathfrak{CCD}}^2}(\mathfrak{C}^{4*}_{04}).$
We are interested only in the cases with 
$(\alpha_6,\alpha_7,\alpha_8) \neq (0,0,0)$, $(\alpha_2,\alpha_4,\alpha_6, \alpha_7) \neq (0,0,0,0)$
and $(\alpha_3,\alpha_5,\alpha_7, \alpha_8) \neq (0,0,0,0).$
Since
 	$$
	\phi_1^T\begin{pmatrix}
	0 & \alpha_1 & \alpha_2 & \alpha_3 \\
	\alpha_1  &  0 & \alpha_4 & \alpha_5 \\
	\alpha_2 & \alpha_4 & \alpha_6   & \alpha_7\\
	\alpha_3 & \alpha_5 & \alpha_7  & \alpha_8
	\end{pmatrix} \phi_1=
	\begin{pmatrix}
	\alpha^* & \alpha_1^* & \alpha^*_2 & \alpha^*_3 \\
	\alpha^*_1  &  \alpha^{**} & \alpha^*_4 & \alpha^*_5 \\
	\alpha^*_2 & \alpha^*_4 & \alpha^*_6   & \alpha^*_7\\
	\alpha^*_3 & \alpha^*_5 & \alpha^*_7  & \alpha^*_8
	\end{pmatrix},
	$$ 
we have
	\begin{longtable}{lcl}
	$\alpha^*_1$ &$=$&$e (x  \af_1+q  \af_4+w  \af_5)+r (x  \af_2+q  \af_6+w  \af_7)+t (x  \af_3+q  \af_7+w  \af_8)$ \\
	$\alpha^*_2$ &$=$&$x^2 (x  \af_2+q  \af_6+w  \af_7)$ \\
	$\alpha^*_3$ &$=$&$e^2 (x  \af_3+q  \af_7+w  \af_8)$ \\
	$\alpha^*_4$ &$=$&$x^2 (e  \af_4+r  \af_6+t  \af_7)$ \\
	$\alpha^*_5$ &$=$&$e^2 (e  \af_5+r  \af_7+t  \af_8)$ \\
	$\alpha^*_6$ &$=$&$x^4  \af_6$\\
	$\alpha^*_7$ &$=$&$e^2 x^2  \af_7$\\
	$\alpha^*_8$ &$=$&$e^4  \af_8$
	\end{longtable}

Then we consider the following cases. 
\begin{enumerate}
    \item If $\alpha_8\neq 0,$   $\alpha_7\neq0$ and $\af_7^2\neq \af_6 \af_8,$ then
     choosing 
    $q=\frac{x (\af_2 \af_8-\af_3 \af_7)}{\af_7^2-\af_6 \af_8},$
    $w=\frac{x (\af_3 \af_6-\af_2 \af_7)}{\af_7^2-\af_6 \af_8},$
    $e=x \sqrt{\frac{\af_7}{\af_8}},$
    $r=\frac{x \sqrt{\af_7} (\af_4 \af_8-\af_5 \af_7)}{\sqrt{\af_8} (\af_7^2-\af_6 \af_8)}$
    and
    $t=\frac{x \sqrt{\af_7} (\af_5 \af_6-\af_4 \af_7)}{\sqrt{\af_8} (\af_7^2-\af_6 \af_8)},$
    we have two families of distinct representatives 
    \[\langle \nabla_1+\af \nabla_6+\nabla_7+\nabla_8 \rangle_{\af\neq 1} \mbox{ and }
    \langle \af \nabla_6+\nabla_7+\nabla_8 \rangle_{\af\neq 1}\] depending on 
    $\af_3 \af_5 \af_6-\af_3 \af_4 \af_7-\af_2 \af_5 \af_7+\af_1 \af_7^2+\af_2 \af_4 \af_8-\af_1 \af_6 \af_8=0$ or not.

\item  If  $\alpha_8\neq 0,$   $\alpha_7\neq0,$   $\af_7^2= \af_6 \af_8$ and 
$ \af_5 \af_7\neq \af_4 \af_8,$ then 
     choosing 
$q=\frac{\af_3 \af_5^2 \af_7^2-\af_1 \af_5 \af_7^2\af_8+\af_1 \af_4\af_7 \af_8^2-\af_2 \af_4 \af_5 \af_8^2}{\sqrt{\af_7^5 \af_8} (\af_5 \af_7-\af_4 \af_8)},$
$w=\frac{\af_2 \af_4 \af_5 \af_8+\af_1 \af_7 (\af_5 \af_7-\af_4 \af_8)+\af_3 \af_4 (-2 \af_5 \af_7+\af_4 \af_8)}{\sqrt{\af_7^3\af_8} (\af_5 \af_7-\af_4 \af_8)},$
$e=\frac{\af_4\af_8-\af_5\af_7}{\af_7 \af_8},$
$r=-\frac{\af_4 \af_5}{\af_7^2},$
$t=\frac{\af_5^2}{\af_8^2}$
and
$x=\frac{\af_4 \af_8-\af_5 \af_7}{\sqrt{\af_7^3\af_8}},$
we have the family of representatives
$\langle \af \nabla_2+ \nabla_4+ \nabla_6+\nabla_7+\nabla_8 \rangle.$
Observe that for $\af \ne \af'$ and $\af \ne 0$ we have 
\[\orb\langle \af \nabla_2+ \nabla_4+ \nabla_6+\nabla_7+\nabla_8 \rangle=
  \orb\langle  \af' \nabla_2+ \nabla_4+ \nabla_6+\nabla_7+\nabla_8 \rangle\] if and only if either $\af'= \pm \af^{-1}$ or   $\af'= - \af.$

\item   If $\alpha_8\neq 0,$   $\alpha_7\neq0,$   $\af_7^2= \af_6 \af_8,$  $ \af_5 \af_7= \af_4 \af_8$
and $\af_3 \af_7 \neq \af_2 \af_8,$ then 
choosing 
    $q=\frac{\af_3^2}{\af_7^2},$
    $w=-\frac{\af_2 \af_3}{\af_7^2},$
    $e=\frac{\af_2 \af_8-\af_3 \af_7}{\sqrt{\af_7^3\af_8}},$
    $r=\frac{\af_3 \af_5-\af_1 \af_8}{\sqrt{\af_7^3\af_8}},$
    $t=\frac{\af_1 \af_7-\af_2 \af_5}{ \sqrt{\af_7^3\af_8}}$
    and 
    $x=\frac{\af_2 \af_8-\af_3 \af_7}{\af_7^2},$ we have the representative
$\langle   \nabla_2+ \nabla_6+\nabla_7+\nabla_8 \rangle$
and after action of some suitable automorphism $\phi_2,$
we have the representative
$\langle   \nabla_4+ \nabla_6+\nabla_7+\nabla_8 \rangle$, which found above.

\item  If  $\alpha_8\neq 0,$   $\alpha_7\neq0,$   $\af_7^2= \af_6 \af_8,$  $ \af_5 \af_7= \af_4 \af_8$
and $\af_3 \af_7 = \af_2 \af_8,$ then choosing 
$q=0,$ $w=-\frac{e \af_3}{\sqrt{\af_7\af_8}},$
$r=0,$ $t=-\frac{e\af_5}{\af_8}$
and 
$x=e \sqrt{\frac{\af_8}{\af_7}},$
we have two representatives 
$\langle   \nabla_1+ \nabla_6+\nabla_7+\nabla_8 \rangle$ and
$\langle \nabla_6+\nabla_7+\nabla_8 \rangle,$
depending in $\af_3 \af_5=\af_1 \af_8,$ or not.
These orbits will be joined to families which found in case (1).

    \item If $\alpha_8\neq 0,$   $\alpha_7= 0$ and $\af_6\neq0,$ then  
     choosing 
    $q=-\frac{x \af_2}{\af_6},$
    $w=-\frac{x \af_3}{\af_8},$
    $e=\frac{x \sqrt[4]{\af_6}}{\sqrt[4]{\af_8}},$
    $r=-\frac{x \af_4}{\sqrt[4]{\af_6^3 \af_8}}$
    and
    $t=-\frac{x \af_5 \sqrt[4]{\af_6}}{\sqrt[4]{\af_8^5}},$
    we have two  representatives 
    $\langle \nabla_1+ \nabla_6+\nabla_8 \rangle$ and 
    $\langle   \nabla_6+\nabla_8 \rangle$ depending on 
    $\af_1 \af_6 \af_8=\af_3 \af_5 \af_6+\af_2 \af_4 \af_8$ or not.

    \item If $\alpha_8\neq 0,$   $\alpha_7= 0$ and  $\af_6=0,$ then  
     choosing 
    $w=-\frac{x \af_3}{\af_8},$
    $t=-\frac{e \af_5}{\af_8},$ we have $\af_3*=\af_5*=0,$ and:
    
    \begin{enumerate}
        \item if $\af_2\neq 0,$ then choosing
        $q=0,$ 
        $e=-\sqrt[4]{\frac{x^3 \af_2}{\af_8}},$
        $r=\frac{x(\af_1 \af_8- \af_3 \af_5)}{\sqrt[4]{x  \af_2^3  \af_8^5}},$
        we have two representatives 
            $\langle \nabla_2+ \nabla_4+\nabla_8 \rangle$ and  
            $\langle \nabla_2+ \nabla_8 \rangle,$ depending on $\af_4=0,$ or not;
            
            \item if $\af_2=0$ and $\af_4\neq 0,$ then  choosing
   $x=\af_4 \sqrt{\af_8},$
            $q=\frac{\af_3 \af_5-\af_1 \af_8}{\sqrt{\af_8}},$
            $e=\af_4$ and $r=0,$ we have   $\langle  \nabla_4+\nabla_8 \rangle;$
             \item if $\af_2=0$ and $\af_4=0,$ then we have split extensions.
            
            \end{enumerate}
      
    \item If $\af_8=0$ and $\af_6\neq 0,$ then by choosing some suitable automorphism $\phi_2$ we have $\af^*_8\neq 0$ and we have a case considered above.
    
    \item If $\af_8=0,$ $\af_7\neq0$ and $\af_6=0,$ then choosing 
    $q=-\frac{x \af_3}{\af_7},$
    $w=-\frac{x \af_2}{\af_7},$
    $r=-\frac{e \af_5}{\af_7}$
    and 
    $t=-\frac{e \af_4}{\af_7}$
    we have two representatives 
    $\langle \nabla_1+ \nabla_7 \rangle$ and   $\langle  \nabla_7 \rangle$ depending on 
    $\af_3 \af_4+\af_2 \af_5\neq \af_1 \af_7$ or not.

\end{enumerate}

Summarizing, all cases we have the following representatives of distinct orbits:
\begin{longtable}{lll}
$\langle \nabla_1+\af \nabla_6+\nabla_7+\nabla_8 \rangle$ & 
$\langle \nabla_1+ \nabla_6+\nabla_8 \rangle$  &
$\langle \nabla_1+ \nabla_7 \rangle$  \\

\multicolumn{3}{l}{$\langle \af \nabla_2+ \nabla_4+ \nabla_6+\nabla_7+\nabla_8 \rangle^{O(\af)=O(-\af)=O(\af^{-1})=O(-\af^{-1}), \ \af\neq 0}$ }\\

$\langle \nabla_2+ \nabla_4+\nabla_8 \rangle$ &
$\langle \nabla_2+ \nabla_8 \rangle$ &
$\langle \nabla_4+ \nabla_8 \rangle$ \\

$\langle \af \nabla_6+\nabla_7+\nabla_8 \rangle$&
$\langle \nabla_6+ \nabla_8 \rangle$ &
$\langle \nabla_7 \rangle$

\end{longtable}

Hence we have the following $5$-dimensional algebras:
\begin{longtable}{llllllllll}
 $\mathfrak{C}_{49}^5(\af)$& $ : $ & $e_1e_1=e_3$ & $e_1e_2=e_5$ &$e_2e_2=e_4$ & $e_3e_3=\af e_5$ & $e_3e_4=e_5$& $e_4e_4=e_5$ \\
 $\mathfrak{C}_{50}^5$& $ : $ & $e_1e_1=e_3$  & $e_1e_2=e_5$ &$e_2e_2=e_4$ & $e_3e_3=e_5$ & $e_4e_4=e_5$\\
 $\mathfrak{C}_{51}^5$& $ : $ & $e_1e_1=e_3$  & $e_1e_2=e_5$ &$e_2e_2=e_4$ & $e_3e_4=e_5$\\
 $\mathfrak{C}_{52}^5(\af)$& $ : $ & $e_1e_1=e_3$  & $e_1e_3=\af e_5$ &$e_2e_2=e_4$ &$e_2e_3=e_5$ &$e_3e_3=e_5$
 &$e_3e_4=e_5$ &$e_4e_4=e_5$\\
 $\mathfrak{C}_{53}^5$& $ : $ & $e_1e_1=e_3$  & $e_1e_3=e_5$ &$e_2e_2=e_4$ & $e_2e_3=e_5$ & $e_4e_4=e_5$\\
 $\mathfrak{C}_{54}^5$& $ : $ & $e_1e_1=e_3$  & $e_1e_3=e_5$ &$e_2e_2=e_4$ & $e_4e_4=e_5$\\
 $\mathfrak{C}_{55}^5$& $ : $ & $e_1e_1=e_3$  & $e_2e_2=e_4$ &$e_2e_3=e_5$ & $e_4e_4=e_5$\\
 $\mathfrak{C}_{56}^5(\af)$& $ : $ & $e_1e_1=e_3$  &$e_2e_2=e_4$ & $e_3e_3=\alpha e_5$ & $e_3e_4=e_5$ & $e_4e_4=e_5$\\
 $\mathfrak{C}_{57}^5$& $ : $ & $e_1e_1=e_3$  &$e_2e_2=e_4$ & $e_3e_3=e_5$ & $e_4e_4=e_5$\\
 $\mathfrak{C}_{58}^5$& $ : $ & $e_1e_1=e_3$  &$e_2e_2=e_4$ & $e_3e_4=e_5$
 \end{longtable}

 \subsection{$1$-dimensional central extensions of ${\mathfrak{C}^{4*}_{05}}$.}
Here we will  collect all information about ${\mathfrak{C}_{05}^{4*}}:$

\begin{longtable}{|l|l|l|l|}
\hline
Algebra  & Multiplication & Cohomology &  Automorphisms  \\
\hline
${\mathfrak{C}}_{05}^{4*}$ &  
$\begin{array}{l}e_1 e_1 = e_3 \\ e_1 e_2=e_4 \end{array}$
& 
$\begin{array}{lcl}
{\rm H^2_{\mathfrak{J}}}({\mathfrak{C}}_{05}^{4*})&=&
\left\langle 
[\Delta_{13}],[\Delta_{14}],[\Delta_{22}],[\Delta_{23}],[\Delta_{24}] \right\rangle \\
{\rm H^2_{\mathfrak{CCD}}}({\mathfrak{C}}_{05}^{4*})&=&
{\rm H^2_{\mathfrak{J}}}({\mathfrak{C}_{05}^{4*}}) \oplus 
\langle[\Delta_{33}],[\Delta_{34}],[\Delta_{44}]\rangle
\end{array}$& 
$\phi=
	\begin{pmatrix}
	x &  0  &  0 & 0\\
	q &  r  &  0 & 0\\
	w &  t  & x^2 & 0\\
	e &  y  & 2xq & xr
	\end{pmatrix} 
	$\\
\hline
\end{longtable}

Let us use the following notations:
\begin{longtable}{llll}
$\nabla_1 = [\Delta_{13}]$ & $\nabla_2 = [\Delta_{14}]$&
$\nabla_3 = [\Delta_{22}]$ & $\nabla_4 = [\Delta_{23}]$\\
$\nabla_5 = [\Delta_{24}]$ & $\nabla_6 = [\Delta_{33}]$&
$\nabla_7 = [\Delta_{34}]$ & $\nabla_8 = [\Delta_{44}]$
\end{longtable}
We are interested only in the cases with 
$(\alpha_6,\alpha_7,\alpha_8) \neq (0,0,0)$, $(\alpha_1,\alpha_4,\alpha_6, \alpha_7) \neq (0,0,0,0)$
and $(\alpha_2,\alpha_5,\alpha_7, \alpha_8) \neq (0,0,0,0).$
	Take $\theta=\sum\limits_{i=1}^8\alpha_{i}\nabla_{i}\in {\rm H_{\mathfrak{CCD}}^2}(\mathfrak{C}^{4*}_{05}).$
	Since
 	$$
	\phi^T\begin{pmatrix}
	0 & 0 & \alpha_1 & \alpha_2 \\
	0  &  \alpha_3 & \alpha_4 & \alpha_5 \\
	\alpha_1 & \alpha_4 & \alpha_6   & \alpha_7\\
	\alpha_2 & \alpha_5 & \alpha_7  & \alpha_8
	\end{pmatrix} \phi=
	\begin{pmatrix}
	\alpha^* & \alpha^{**} & \alpha^*_1 & \alpha^*_2 \\
	\alpha^{**}  &  \alpha^{*}_3 & \alpha^*_4 & \alpha^*_5 \\
	\alpha^*_1 & \alpha^*_4 & \alpha^*_6   & \alpha^*_7\\
	\alpha^*_2 & \alpha^*_5 & \alpha^*_7  & \alpha^*_8
	\end{pmatrix},
	$$ 
we have
\begin{longtable}{lcl}
$\af^*_1$ & $ = $ & $  x (x^2 \af_1+x (q (2 \af_2+\af_4)+w \af_6+e \af_7)+2 q (q \af_5+w \af_7+e \af_8))$\\
$\af^*_2$ & $ = $ & $ r x (x \af_2+q \af_5+w \af_7+e \af_8)$\\
$\af^*_3$ & $ = $ & $ r^2 \af_3+2 r (t \af_4+y \af_5)+t^2 \af_6+2 t y \af_7+y^2 \af_8$\\
$\af^*_4$ & $ = $ & $ x (r x \af_4+2 q r \af_5+t x \af_6+2 q t \af_7+x y \af_7+2 q y \af_8)$\\
$\af^*_5$ & $ = $ & $ r x (r \af_5+t \af_7+y \af_8)$\\
$\af^*_6$ & $ = $ & $ x^2 (x^2 \af_6+4 q x \af_7+4 q^2 \af_8)$\\
$\af^*_7$ & $ = $ & $ r x^2 (x \af_7+2 q \af_8)$\\
$\af^*_8$ & $ = $ & $ r^2 x^2 \af_8$
\end{longtable}

Then we consider the following cases.
\begin{enumerate}
    \item If $\af_8\neq 0$ and $\af_7^2\neq\af_6 \af_8,$ then choosing 
    $q=-\frac{x \af_7}{2 \af_8},$
    $w=\frac{x (\af_5 \af_7^2+\af_8 (-2 \af_2 \af_7-\af_4 \af_7+2 \af_1 \af_8))}{2 \af_8 (\af_7^2-\af_6 \af_8)},$
    $e=\frac{x (\af_5 \af_6 \af_7-\af_4 \af_7^2-2 \af_2 \af_6 \af_8+2 \af_1 \af_7 \af_8)}{2 \af_8 (-\af_7^2+\af_6 \af_8)},$
    $r=\frac{x \sqrt{\af_6 \af_8-\af_7^2}}{\af_8},$
    $t=\frac{-r \af_5 \af_7+r \af_4 \af_8+\sqrt{r^2 \af_8 (\af_5^2 \af_6-2 \af_4 \af_5 \af_7+\af_3 \af_7^2+\af_4^2 \af_8-\af_3 \af_6 \af_8)}}{\af_7^2-\af_6 \af_8}$
    and
    $y=\frac{-r \af_5 \af_6 \af_8+r \af_4 \af_7 \af_8+\af_7\sqrt{r^2 \af_8 (\af_5^2 \af_6-2 \af_4 \af_5 \af_7+\af_3 \af_7^2+\af_4^2 \af_8-\af_3 \af_6 \af_8)}}{\af_8 (-\af_7^2+\af_6 \af_8)},$
    we have two representatives
    $\langle \nabla_4+\nabla_6+\nabla_8 \rangle$ and  $\langle \nabla_6+\nabla_8 \rangle,$
    depending on $\af_5^2 \af_6-2 \af_4 \af_5 \af_7+\af_4^2 \af_8+\af_3 (\af_7^2-\af_6 \af_8)=0$ or not.
    Note that $\langle \nabla_6+\nabla_8\rangle$ gives a trivial extension.

    \item If $\af_8\neq 0,$ $\af_7^2=\af_6 \af_8$ and 
    $\af_5 \af_7\neq\af_4 \af_8,$
    then choosing 
    $q=-\frac{x \af_7}{2 \af_8},$
    $e=\frac{x \af_5 \af_7-2 x \af_2 \af_8-2 w \af_7 \af_8}{2 \af_8^2},$
    $r=\frac{\af_4 \af_8-\af_5 \af_7}{\af_8^2},$
    $t=\frac{\af_5^2-\af_3 \af_8}{2 \af_8^2}$ 
    and
    $y=\frac{\af_5^2 \af_7-2 \af_4 \af_5 \af_8+\af_3 \af_7 \af_8}{2 \af_8^3},$
    we have two representatives
    $\langle \nabla_1+\nabla_4+\nabla_8 \rangle$ and  $\langle \nabla_4+\nabla_8 \rangle,$
    depending on $\af_5 \af_7^2+\af_8 (-2 \af_2 \af_7-\af_4 \af_7+2 \af_1 \af_8)=0$ or not.

    \item If $\af_8\neq 0,$ $\af_7^2=\af_6 \af_8,$  
    $\af_5 \af_7=\af_4 \af_8$ and $\af_5^2\neq\af_3 \af_8,$
    then choosing 
    $q=-\frac{\af_7 \sqrt{-\af_5^2+\af_3 \af_8}}{2 \af_8^2},$
    $w=0,$
    $e=\frac{(\af_5 \af_7-2 \af_2 \af_8) \sqrt{\af_3 \af_8-\af_5^2}}{2 \af_8^3},$
    $t=0,$
    $y=-\frac{r \af_5}{\af_8}$
    and 
    $x=\frac{\sqrt{\af_3 \af_8-\af_5^2}}{\af_8},$
    we have two representatives 
       $\langle \nabla_1+\nabla_3+\nabla_8 \rangle$ and  $\langle \nabla_3+\nabla_8 \rangle,$
    depending on $\af_2 \af_7\neq\af_1 \af_8$ or not. Note that $\langle \nabla_3+\nabla_8\rangle$ gives a trivial extension.

   \item If $\af_8\neq 0,$ $\af_7^2=\af_6 \af_8,$  
    $\af_5 \af_7=\af_4 \af_8$ and $\af_5^2=\af_3 \af_8,$
    then choosing 
    $q=-\frac{\af_7}{2 \af_8},$
    $w=0,$
    $e=\frac{ \af_5 \af_7-2  \af_2 \af_8}{2 \af_8^2},$
    $t=0,$
    $y=-\frac{r \af_5}{\af_8}$
    and 
    $x=1,$
    we have two representatives 
       $\langle \nabla_1+ \nabla_8 \rangle$ and  $\langle \nabla_8 \rangle,$
    depending on $\af_2 \af_7\neq\af_1 \af_8$ or not.
      Note that $\langle  \nabla_8\rangle$ gives a trivial extension.

    \item If $\af_8=0$ and $\af_7\neq 0,$ then  choosing 
    $q=-\af_6 \af_7^2,$
    $w=\af_7 (\af_5 \af_6-4 \af_2 \af_7),$
    $e=-\af_5 \af_6^2+\af_7 (4 \af_2 \af_6+\af_4 \af_6-4 \af_1 \af_7),$
    $t=-\frac{r \af_5}{\af_7},$
    $y=\frac{r (\af_5 \af_6-\af_4 \af_7)}{\af_7^2}$
    and
    $x=4 \af_7^3,$
    we have two representatives 
       $\langle \nabla_3+ \nabla_7 \rangle$ and  $\langle \nabla_7 \rangle,$
    depending on $\af_5^2 \af_6+\af_3 \af_7^2\neq 2 \af_4 \af_5 \af_7$ or not.
  
    \item If  $\af_8=0,$ $\af_7=0,$ $\af_6\neq 0$ and $\af_5\neq0,$
    then choosing 
    $q=-\af_2 \af_6,$
    $w=\af_2 \af_4-\af_1 \af_5,$
    $e=0,$
    $r=\af_5 \af_6^2,$
    $t=(2 \af_2-\af_4)\af_5 \af_6,$
    $y=\frac{\af_6(-4 \af_2^2+\af_4^2-\af_3 \af_6)}{2}$
    and 
    $x=\af_5 \af_6,$ we have the representative $\langle \nabla_5+ \nabla_6 \rangle.$
    
    \item If $\af_8=0,$ $\af_7=0,$ $\af_6\neq 0$ and $\af_5=0,$
    then choosing 
    $q=0,$
    $w=-\frac{x \af_1}{\af_6},$
    $e=0,$
    $y=0$
    and
    $t=-\frac{r \af_4}{\af_6},$
    we have three types of representatives
        $\langle \nabla_2+ \af \nabla_3+ \nabla_6 \rangle,$   $\langle \nabla_3 +\nabla_6 \rangle$
        or $\langle  \nabla_6 \rangle$
    depending on $\af_4^2=\af_3 \af_6$ and $\af_2=0$ or not.
    Note that $\langle  \nabla_6\rangle$ gives a trivial extension.

\end{enumerate}

Summarizing, all cases we have the following representatives of distinct orbits:
\begin{longtable}{lllll}
$\langle \nabla_1+ \nabla_3+\nabla_8 \rangle$ &
$\langle \nabla_1+ \nabla_4+\nabla_8 \rangle$ & 
$\langle \nabla_1+ \nabla_8 \rangle$   &
$\langle \nabla_2+ \af \nabla_3+ \nabla_6 \rangle$   &
$\langle \nabla_3+ \nabla_7 \rangle$  \\ 
$\langle \nabla_4+ \nabla_6+\nabla_8 \rangle$ &
$\langle \nabla_4+ \nabla_8 \rangle$&
$\langle \nabla_5+ \nabla_6 \rangle$&
$\langle \nabla_6+ \nabla_8 \rangle$&
$\langle \nabla_7  \rangle$
\end{longtable}

Hence we have the following $5$-dimensional algebras:
\begin{longtable}{llllllllll}
$\mathfrak{C}_{59}^5$& $ : $ & $e_1e_1=e_3$ & $e_1e_2=e_4$ &$e_1e_3=e_5$ &$e_2e_2=e_5$ &$e_4e_4=e_5$\\
$\mathfrak{C}_{60}^5$& $ : $ & $e_1e_1=e_3$ & $e_1e_2=e_4$ &$e_1e_3=e_5$ &$e_2e_3=e_5$ &$e_4e_4=e_5$\\
$\mathfrak{C}_{61}^5$& $ : $ & $e_1e_1=e_3$ & $e_1e_2=e_4$ &$e_1e_3=e_5$ &$e_4e_4=e_5$\\ 
$\mathfrak{C}_{62}^5(\af)$& $ : $ & $e_1e_1=e_3$ & $e_1e_2=e_4$&$e_1e_4=e_5$ &$e_2e_2=\af e_5$ &$e_3e_3=e_5$\\ 
$\mathfrak{C}_{63}^5$& $ : $ & $e_1e_1=e_3$ & $e_1e_2=e_4$ &$e_2e_2=e_5$ &$e_3e_4=e_5$\\
$\mathfrak{C}_{64}^5$& $ : $ & $e_1e_1=e_3$ & $e_1e_2=e_4$ &$e_2e_3=e_5$ &$e_3e_3=e_5$ &$e_4e_4=e_5$\\ 
$\mathfrak{C}_{65}^5$& $ : $ & $e_1e_1=e_3$ & $e_1e_2=e_4$ &$e_2e_3=e_5$ &$e_4e_4=e_5$\\
$\mathfrak{C}_{66}^5$& $ : $ & $e_1e_1=e_3$ & $e_1e_2=e_4$ &$e_2e_4=e_5$ &$e_3e_3=e_5$\\ 
$\mathfrak{C}_{67}^5$& $ : $ & $e_1e_1=e_3$ & $e_1e_2=e_4$ &$e_3e_3=e_5$ &$e_4e_4=e_5$\\ 
$\mathfrak{C}_{68}^5$& $ : $ & $e_1e_1=e_3$ & $e_1e_2=e_4$ &$e_3e_4=e_5$
\end{longtable}

  \subsection{$1$-dimensional central extensions of ${\mathfrak{C}_{06}^{4*}}$.}
Here we will  collect all information about ${\mathfrak{C}_{06}^{4*}}:$

\begin{longtable}{|l|l|l|l|}
\hline
Algebra  & Multiplication & Cohomology &  Automorphisms  \\
\hline
${\mathfrak{C}}_{06}^{4*}$ &  
$\begin{array}{l}e_1 e_1 = e_4 \\ e_2 e_3=e_4 \end{array}$
& 
$\begin{array}{lcl}
{\rm H^2_{\mathfrak{J}}}({\mathfrak{C}}_{06}^{4*})&=&
\langle [\Delta_{ij}]  \rangle_{1\leq i \leq j \leq 4}^{ (i,j) \not \in \{ (1,1), \ (4,4) \} } \\
{\rm H^2_{\mathfrak{CCD}}}({\mathfrak{C}}_{06}^{4*})&=&{\rm 
H^2_{\mathfrak{J}}}({\mathfrak{C}_{06}^{4*}}) \oplus 
\langle  [\Delta_{44}]\rangle
 \end{array}$& 
$\phi=
	\begin{pmatrix}
	x &  a  & p & 0\\
	y &  b  & q & 0\\
	z &  c  & r &  0\\
	t &  d  & s & x^2+2yz
	\end{pmatrix} 
	$\\
\hline
\multicolumn{4}{|c|}{
	where $xa+yc+zb=0,\ xp+yr+zq=0,\ a^2+2bc=0,\ p^2+2qr=0,\ x^2+2yz=ap+br+cq$ 
}\\
\hline
\end{longtable}

Let us use the following notations:
\begin{longtable}{lll}
$\nabla_1 = [\Delta_{12}]$ & $\nabla_2 = [\Delta_{13}]$ & $\nabla_3 = [\Delta_{14}]$\\ 
$\nabla_4 = [\Delta_{22}]$ & $\nabla_5 = [\Delta_{23}]$ & $\nabla_6 = [\Delta_{24}]$\\ 
$\nabla_7 = [\Delta_{33}]$ & $\nabla_8 = [\Delta_{34}]$ & $\nabla_9 = [\Delta_{44}]$
\end{longtable}
Take $\theta=\sum\limits_{i=1}^9\alpha_{i}\nabla_{i}\in {\rm H_{\mathfrak{CCD}}^2}(\mathfrak{C}_{06}^{4*}).$
We are interested only in $\alpha_9\neq 0$ and we suppose that $\af_9=1.$
Since
	$$
	\phi^T\begin{pmatrix}
	0 & \alpha_1 & \alpha_2 & \alpha_3 \\
	\alpha_1  &  \alpha_4 & \alpha_5 & \alpha_6 \\
	\alpha_2 & \alpha_5 & \alpha_7   & \alpha_8\\
	\alpha_3 & \alpha_6 & \alpha_8  & \alpha_9
	\end{pmatrix} \phi=
	\begin{pmatrix}
	\alpha^* & \alpha_1 & \alpha^*_2 & \alpha^*_3 \\
	\alpha^*_1  &  \alpha^{*}_4 & \alpha^{*}_5+\alpha^{*} & \alpha^*_6 \\
	\alpha^*_2 & \alpha^{*}_5+\alpha^{*} & \alpha^*_7   & \alpha^*_8\\
	\alpha^*_3 & \alpha^*_6 & \alpha^*_8  & \alpha^*_9
	\end{pmatrix},
	$$ 
we have
	\begin{longtable}{lcl}
$\af^*_1 $ & $ = $ & $x (b \af_1+c \af_2+d \af_3)+y (a \af_1+b \af_4+c \af_5+d \af_6)+$\\
&&$t (d+a \af_3+b \af_6+c \af_8)+z (a \af_2+b \af_5+c \af_7+d \af_8)$	\\ 

$\af^*_2 $ & $ = $ & $x (q \af_1+r \af_2+s \af_3)+y (p \af_1+q \af_4+r \af_5+s \af_6)+$\\
&&$t (s+p \af_3+q \af_6+r \af_8)+z (p \af_2+q \af_5+r \af_7+s \af_8)$\\

$\af^*_3 $ & $ = $ & $(x^2+2 y z) (t+x \af_3+y \af_6+z \af_8)$\\

$\af^*_4 $ & $ = $ & $d^2+2 a (b \af_1+c \af_2)+b^2 \af_4+2 b c \af_5+c^2 \af_7+2 d (a \af_3+b \af_6+c \af_8)$\\

$\af^*_5 $ & $ = $ & $-t^2+a (q \af_1+r \af_2+s \af_3)-2 x (y \af_1+z \af_2+t \af_3)-y^2 \af_4+b (p \af_1+q \af_4+r \af_5+s \af_6)-$\\
&&$2 y (z \af_5+t \af_6)-z^2 \af_7-2 t z \af_8+d (s+p \af_3+q \af_6+r \af_8)+c (p \af_2+q \af_5+r \af_7+s \af_8)$\\

$\af^*_6 $ & $ = $ & $(x^2+2 y z) (d+a \af_3+b \af_6+c \af_8)$\\

$\af^*_7 $ & $ = $ & $s^2+2 p (q \af_1+r \af_2)+q^2 \af_4+2 q r \af_5+r^2 \af_7+2 s (p \af_3+q \af_6+r \af_8)$\\

$\af^*_8 $ & $ = $ & $(x^2+2 y z) (s+p \af_3+q \af_6+r \af_8)$\\

$\af^*_9 $ & $ = $ & $(x^2+2 y z)^2.$
	\end{longtable}

Then by  choosing 
$d = -a\af_3 - b\af_6 - c\af_8,$
$t = -x\af_3 - y\af_6 - z\af_8,$
$s= -p\af_3 - q\af_6 - r\af_8$ and some suitable $a,b,c,d,x,y,z,p,q,s$ we have that 
$\af^*_3=0$, $\af_6^*=0$ and $\af_8^*=0.$ Now, we can suppose that 
$\af_3=0$, $\af_6=0,$  $\af_8=0.$
If  $\af_4 \neq 0$ and 

\begin{enumerate}
    \item  $(\af_1,\af_2,\af_5,\af_7)\neq (0,0,0,0),$ then choosing 
$a=-\frac{X}{2},$
$b=-\frac{1}{2},$
$c=\frac{X^2}{4},$
$p=-X,$
$q=1,$
$r=-\frac{X^2}{2},$
$x=0,$
$y=1,$
$z=\frac{X^2}{2},$
where $\af_7X^4-4  \af_2X^3-4  \af_5X^2+8 \af_1 X +4 \af_4=0,$
and 
$t=d=s=0,$
we have representative with $\af_4^*=0$ and $\af_3^*=\af_6^*=\af^*_8=0.$

   \item  $(\af_1,\af_2,\af_5,\af_7)= (0,0,0,0),$ then choosing 
   $a=b=p=r=y=z=t=d=s=0$ and $c=q=x=1,$ we have representative with $\af_4^*=0$ and $\af_3^*=\af_6^*=\af^*_8=0.$
\end{enumerate}

Now we can suppose that $\af_4=0,$ then

\begin{enumerate}
  
 \item if $\af_1\neq 0$   and $18 \af_1 \af_2 \af_5+2 \af_5^3\neq 27 \af_1^2 \af_7,$
 then  choosing
 $a=-\frac{\af_5}{3 x},$ 
 $b=-\frac{\af_5^2}{18 x \af_1},$
 $c=\frac{\af_1}{x},$
 $p=0,$
 $q=\frac{x^3}{\af_1},$
 $r=0,$
 $y=\frac{x \af_5}{3 \af_1},$
 $z=t=d=s=0$
 and 
 $x= 3^{-\frac{1}{2}} (-18 \af_1 \af_2 \af_5 - 2 \af_5^3 + 
 27 \af_1^2 \af_7)^{\frac{1}{6}},$ we have the family of representatives
 $\langle \af \nabla_1+ \nabla_2 +\nabla_4+\nabla_9\rangle;$
    
  \item if $\af_1\neq 0$   and $18 \af_1 \af_2 \af_5+2 \af_5^3= 27 \af_1^2 \af_7$
 then choosing   $a=-\frac{\af_5}{3 x},$  
 $b=-\frac{\af_5^2}{18 x \af_1},$
 $c=\frac{\af_1}{x},$
 $p=0,$
 $q=\frac{x^3}{\af_1},$
 $r=0,$
 $y=\frac{x \af_5}{3 \af_1}$
 and
 $z=t=d=s=0,$
 we have two representative 
  $\langle  \nabla_1+ \nabla_2  +\nabla_9\rangle$ and   $\langle   \nabla_2  +\nabla_9\rangle$
 depending on $6\af_1 \af_2 + \af_5^2=0$ or not;
 
 \item if $\af_1=0$ and $\af_5\neq 0,$
 then choosing 
 $a=\frac{c \af_2}{\af_5},$
 $b=-\frac{c \af_2^2}{2 \af_5^2},$
 $p=0,$
 $q=\frac{\af_5}{c},$
 $r=0,$
 $x=\sqrt{\af_5},$
 $y=-\sqrt{\af_2\af_5}$
 and
 $z=t=d=s=0,$  we have two representative 
  $\langle  \nabla_4+ \nabla_5  +\nabla_9\rangle$ and   $\langle   \nabla_5  +\nabla_9\rangle$
 depending on $\af_2^2+\af_5 \af_7=0$ or not.

\item if $\af_1=0$ and $\af_5= 0,$
then choosing 
 $b=-\frac{a^2}{2 c},$
 $p=0,$
 $q=\frac{1}{c},$
 $r=0,$
 $x=1,$
 $y=-\frac{a}{c}$
 and
 $z=t=d=s=0,$  we have two representative 
  $\langle  \nabla_1+ \nabla_9\rangle$,    $\langle   \nabla_4  +\nabla_9\rangle$ and 
    $\langle \nabla_9\rangle$
 depending on $\af_2=0$ and $\af_7=0$ or not.
Note that, the representative $\langle  \nabla_1+ \nabla_9\rangle$ under the action of an automorphism
with $a=b=p=r=y=z=t=d=s=0$ and $c=q=x=1$ gives  $\langle  \nabla_2+ \nabla_9\rangle.$

\end{enumerate}

Summarizing, we have the family of orbits
\[\langle \af \nabla_1+ \nabla_2 +\nabla_4+\nabla_9\rangle\]
and the following separated orbits:
 \begin{longtable}{lll}
$\langle \nabla_1+ \nabla_2  +\nabla_9\rangle$ &
$\langle \nabla_1+ \nabla_9\rangle$    &
$\langle \nabla_4+ \nabla_5  +\nabla_9\rangle$  \\
$\langle \nabla_4  +\nabla_9\rangle$  &
$\langle \nabla_5  +\nabla_9\rangle$ &
$\langle \nabla_9 \rangle$
\end{longtable}
It is easy to prove that these separated orbits are distinct and non-isomorphic to elements from the family.

In the family of orbits, all orbits are distinct, excepting
\[\langle \af \nabla_1+ \nabla_2 +\nabla_4+\nabla_9\rangle =
\langle \sqrt[3]{1} \af \nabla_1+ \nabla_2 +\nabla_4+\nabla_9\rangle \]

For prove it, we will suppose that there is an automorshism $\phi$ such that 
\[\phi \langle \af \nabla_1+ \nabla_2 +\nabla_4+\nabla_9\rangle =
\langle  \bt \nabla_1+ \nabla_2 +\nabla_4+\nabla_9\rangle \]
It is easy to see that $s=d=t=0$ and we consider the following cases:
\begin{enumerate}
    \item $q=0,$ then subsequently we have 
\[ p = 0, \ y = 0, \   z = 0, \ a = 0, \   c = 0, \ b = \frac{x^2}{r},   \ x^3 = r, \ r=\pm1 \mbox{ and }\af=\bt.\]

    \item $b=0,$ then $a=0$  and $\af_4^*=0,$ which gives a contradiction.
    
    \item $q\neq 0,$ $b\neq0$ and $y=0,$ then   subsequently we have 
    \[ p = -\frac{q z}{x}, \ 
a= -\frac{b z}{x}, \
r = -\frac{q z^2}{2 x^2}, \
c = -\frac{b z^2}{2 x^2},\]
which gives a contradiction, because $\det \phi =0.$ 
    
\item $q\neq 0,$  $b\neq 0$ and $y\neq0,$ then   subsequently we have 
\[ r = -\frac{p^2}{2 q}, \ c = -\frac{a^2}{2 b}, \ p = X_1 q,\ z =   \frac{X_1^2 y}{2}-x X_1, \ 
a = b (\frac{2 x}{y}-X_1), \ b = -\frac{y^2}{2 q}, \ x = X_2 y.\]
Within these conditions we have  
\begin{longtable}{lll}
$x a+y c+z b=0$ & $x p+y r+z q=0$ & $a^2+2b c=0$\\
$p^2+2q r=0$ & \multicolumn{2}{l}{$x^2+2y z-(a p+b r+c q)=0$}\\
 $\af^*_3=0$ &
$\af^*_6=0$ &
$\af^*_8=0$ 
\end{longtable}
and 
\begin{longtable}{lcl}
$\af^*_1 $ & $ = $ & $\frac{y^3 (-2+X_1^3-3 X_1^2 X_2+4 X_2^3+2 X_1 \af-6 X_2 \af)}{4 q}$\\
$\af^*_2 $ & $ = $ & $\frac{q y (2+X_1^3-3 X_1^2 X_2+2 X_1 \af+2 X_2 \af)}{2}$\\
$\af^*_4 $ & $ = $ & $\frac{y^4 (1+X_1^3-6 X_1^2 X_2+12 X_1 X_2^2-8 X_2^3-2 X_1 \af+4 X_2 \af)}{4 q^2}$\\
$\af^*_5 $ & $ = $ & $-\frac{3 y^2 (1+X_1^2 X_2-2 X_1 X_2^2+2 X_2 \af)}{2}$\\
$\af^*_7 $ & $ = $ & $q^2 (1-X_1^3+2 X_1 \af)$\\
$\af^*_9 $ & $ = $ & $(X_1-X_2)^4 y^4$

\end{longtable}
    \end{enumerate}

It follows, that 
\[1-X_1^3+2 X_1 \af=0 \mbox{ and }1+X_1^2 X_2-2 X_1 X_2^2+2 X_2 \af=0.\]
Analyzing the first equation, we can take $X_1,$ such that $ X_1^3\neq -\frac{1}{2}.$ 
On the other side, we have a contradiction with the fact that the product of all roots of the first equation is equal to $1.$
And also, from the second equation we can take $X_2,$ such that $X_1\neq X_2.$ 
On the other side, we will have a contradiction with $ X_1^3\neq -\frac{1}{2}.$ 
Now, it is easy to see, that
\[ \af = \frac{ X_1^2 + 2 X_1 X_2}{2} \mbox{ and }X_2 = -\frac{1}{2 X_1^2}.\]
It follows, that 
\begin{longtable}{lll}
$\af^*_1=\frac{(-1-3 X_1^3+4 X_1^9) y^3}{8 q X_1^6}$ & 
$\af^*_2=\frac{q (1+2 X_1^3)^2 y}{4 X_1^3}$ 
&$\af^*_4=\frac{(1+2 X_1^3)^2 y^4}{4 q^2 X_1^6}$\\
$\af_5^*=0$ &
$\af_7^*=0$ &
$\af^*_9=\frac{(y+2 X_1^3 y)^4}{16 X_1^8}$

\end{longtable}

For receiving $\af_2^*=\af_4^*=\af_9^*$ we should choose
\[q =  \frac{2 \sqrt{1} X_1}{1 + 2 X_1^3} \mbox{ and }  
y = \frac{2 \sqrt[3]{1} \sqrt{1}  X_1^2 }{1 + 2 X_1^3}\]
and from here we have $\beta = \sqrt[3]{1} \af.$

Hence we have the following $5$-dimensional algebras:

\begin{longtable}{lllllllllllll}
${\mathfrak{C}}_{69}^{5}(\af)$ &  
$e_1 e_1 = e_4$&$e_1e_2=\af e_5$ &$e_1e_3=e_5$ &$e_2e_2=e_5$
&$ e_2 e_3=e_4$ &$e_4e_4=e_5$\\

${\mathfrak{C}}_{70}^{5}$ &  
$e_1 e_1 = e_4$&$e_1e_2=e_5$ &$e_1e_3=e_5$ &$ e_2 e_3=e_4$ &$e_4e_4=e_5$ \\

${\mathfrak{C}}_{71}^{5}$ &  
$e_1 e_1 = e_4$&$e_1e_2=e_5$&$ e_2 e_3=e_4$ &$e_4e_4=e_5$ \\

${\mathfrak{C}}_{72}^{5}$ &  
$e_1 e_1 = e_4$&$e_2e_2=e_5$&$ e_2 e_3=e_4+e_5$ &$e_4e_4=e_5$ \\

${\mathfrak{C}}_{73}^{5}$ &  
$e_1 e_1 = e_4$&$e_2e_2=e_5$&$ e_2 e_3=e_4$ &$e_4e_4=e_5$ \\

${\mathfrak{C}}_{74}^{5}$ &  
$e_1 e_1 = e_4$&$ e_2 e_3=e_4+e_5$ &$e_4e_4=e_5$ \\

${\mathfrak{C}}_{75}^{5}$ &  
$e_1 e_1 = e_4$&$ e_2 e_3=e_4$ &$e_4e_4=e_5$ \\

\end{longtable}
where 
${\mathfrak{C}}_{69}^{5}(\af) \cong  {\mathfrak{C}}_{69}^{5}(\sqrt[3]{1}\af).$

  \subsection{$1$-dimensional central extensions of ${\mathfrak{C}_{08}^{4*}}$.}
Here we will  collect all information about ${\mathfrak{C}_{08}^{4*}}:$

\begin{longtable}{|l|l|l|}
\hline
Algebra  & Multiplication & Cohomology   \\
\hline
${\mathfrak{C}}_{08}^{4*}$ &  
$\begin{array}{l}e_1 e_1 = e_2 \\ e_1 e_2=e_4 \\ e_3e_3=e_4 \end{array}$
& 
$\begin{array}{lcl}
{\rm H^2_{\mathfrak{J}}}({\mathfrak{C}}_{08}^{4*})&=&
\langle  
[\Delta_{13}],[\Delta_{14}+\Delta_{22}], 
[\Delta_{23}],[\Delta_{33}]  \rangle \relax \\
{\rm H^2_{\mathfrak{CCD}}}({\mathfrak{C}}_{08}^{4*})&=&{\rm 
H^2_{\mathfrak{J}}}({\mathfrak{C}_{08}^{4*}}) \oplus 
\langle  [\Delta_{14}]\rangle
 \end{array}$\\
\hline
\multicolumn{3}{|l|}{ Automorphisms}\\
\hline	
\multicolumn{3}{|c|}{$\phi=
	\begin{pmatrix}
	x &  0    & 0 & 0\\
	y &  x^2  & -\frac{z r}{x} & 0\\
	z &  0    & r &  0\\
	t &  z^2+2xy  & s & x^3
	\end{pmatrix}, \mbox{ where }  r^2=x^3$}
	\\

\hline
\end{longtable}
Let us use the following notations:
\begin{longtable}{lllll}
$\nabla_1 = [\Delta_{13}]$& 
$\nabla_2 = [\Delta_{14}+\Delta_{22}]$& 
$\nabla_3 = [\Delta_{23}]$&	
$\nabla_4 = [\Delta_{33}]$&
$\nabla_5 = [\Delta_{14}]$ 
\end{longtable}

	Take $\theta=\sum\limits_{i=1}^5\alpha_{i}\nabla_{i}\in {\rm H_{\mathfrak{CCD}}^2}(\mathfrak{C}_{08}^{4*}).$
We are interested only in $\alpha_5\neq 0$ and 
$\alpha_2\neq -\af_5,$ then can suppose that $\alpha_5=1.$ 
	Since

	$$
	\phi^T\begin{pmatrix}
	0 & 0 & \alpha_1 & \alpha_2+1 \\
	0  &  \alpha_2 & \alpha_3 & 0 \\
	\alpha_1 & \alpha_3 & \alpha_4& 0\\
	\alpha_2+1 & 0 & 0  & 0
	\end{pmatrix} \phi=
	\begin{pmatrix}
	\alpha^{**} &  \alpha^{*} & \alpha^*_1 & \alpha^*_2+\af_5^* \\
	\alpha^{*}  &  \alpha^{*}_2 & \alpha^{*}_3 & 0 \\
	\alpha^*_1 &   \alpha^{*}_3 & \alpha^*_4+\alpha^{*}  & 0\\
	\alpha^*_2+\af_5^* & 0 & 0  & 0
	\end{pmatrix},
	$$ 
we have
	\begin{longtable}{lcl}
$\alpha^*_1$ &$=$&$s x (1+\af_2)- rzx^{-1} (y \af_2+z \af_3) + r(x \af_1+y \af_3+z \af_4)$\\
$\alpha^*_2$ &$=$&$x^4 \af_2$ \\
$\alpha^*_3$ &$=$&$r x (-z \af_2+x \af_3)$\\
$\alpha^*_4$ &$=$&$x (x^2 \af_4-z^2 - xy (2 + 3 \af_2) - 3 x z \af_3) $\\
$\alpha^*_5$ &$=$&$x^4$. 
	\end{longtable}

Then we have the following cases:
\begin{enumerate}
    \item if $\alpha_2\neq 0$ and $\af_2 \neq -\frac{2}{3},$ 
    then choosing 
    $x=1,$ $y=\frac{-(1+3 \af_2) \af_3^2+\af_2^2 \af_4}{\af_2^2 (2+3 \af_2)},$
    $z=\frac{\af_3}{\af_2},$
    $t=0,$
    $s=\frac{-\af_1 \af_2^2+\af_3^3-\af_2 \af_3 \af_4}{\af_2^2 (1+\af_2)}$
    and 
    $r=1,$
    we have the family of representative $\langle \af \nabla_2+\nabla_5 \rangle_{\af \not\in \{ -1,- \frac{2}{3},0\} };$
    
    \item if $\af_2 = -\frac{2}{3},$ 
    then choosing
    $y=0,$ $z=-\frac{3 x \af_3}{2},$ 
    $t=0,$ $s=-\frac{3 r (4 \af_1-9 \af_3^3-6 \af_3 \af_4)}{4}$
    and $r^2=x^3,$
    we have two representatives 
   $\langle  -\frac{2}{3} \nabla_2+\nabla_5 \rangle $ and
      $\langle  -\frac{2}{3} \nabla_2+\nabla_4+\nabla_5 \rangle,$ depending in  $ 9 \af_3^2+4 \af_4=0$ or not;

    \item if $\af_2 =0,$ 
    then choosing
    $y=\frac{x \af_4}{2},$ $z=0,$ 
    $t=0,$ $s=-\frac{r (2 \af_1+\af_3 \af_4)}{2}$
    and $r^2=x^3,$
    we have two representatives 
   $\langle   \nabla_5 \rangle $ and 
      $\langle  \nabla_3+\nabla_5 \rangle,$ depending in  $\af_3=0$ or not.
      
\end{enumerate}

Summarizing, we have the following distinct orbits
\begin{longtable}{lll}
$\langle \af \nabla_2+\nabla_5 \rangle_{\af \neq  -1}$ & 
$\langle  - 2 \nabla_2+3\nabla_4+3\nabla_5 \rangle$& 
$\langle  \nabla_3+\nabla_5 \rangle$ 
\end{longtable}

Hence we have the following $5$-dimensional algebras:
\begin{longtable}{llllllllll}
$\mathfrak{C}_{16}^5(\alpha \neq-1)$& $ : $ & $e_1e_1=e_2$ & $e_1e_2=e_4$ &\multicolumn{2}{l}{$e_1e_4= (\af+1) e_5$}& $e_2e_2=\af e_5$ &$e_3e_3=e_4$\\
$\mathfrak{C}_{76}^5$& $ : $ & $e_1e_1=e_2$ & $e_1e_2=e_4$  &$e_1e_4= e_5$ &$e_2e_2= - 2 e_5$ &$e_3e_3=e_4+3e_5$\\
$\mathfrak{C}_{77}^5$& $ : $ & $e_1e_1=e_2$ & $e_1e_2=e_4$ &$e_1e_4= e_5$ &$e_2e_3=  e_5$ &$e_3e_3=e_4$
\end{longtable}

 \subsection{$1$-dimensional central extensions of ${\mathfrak{C}^{4*}_{09}}$.}
Here we will  collect all information about ${\mathfrak{C}_{09}^{4*}}:$

\begin{longtable}{|l|l|l|l|}
\hline
Algebra  & Multiplication & Cohomology &  Automorphisms  \\
\hline
${\mathfrak{C}}_{09}^{4*}$ &  
$\begin{array}{l}
e_1 e_1 = e_2 \\ e_1 e_2=e_3\\
e_1e_3=e_4 \\ e_2e_2=e_4 \end{array}$
& 
$\begin{array}{lcl}
{\rm H^2_{\mathfrak{J}}}({\mathfrak{C}}_{09}^{4*})&=&
\left\langle 
[\Delta_{14}]+[\Delta_{23}] \right\rangle \\
{\rm H^2_{\mathfrak{CCD}}}({\mathfrak{C}}_{09}^{4*})&=&{\rm H^2_{\mathfrak{J}}}({\mathfrak{C}_{09}^{4*}}) \oplus 
\langle [\Delta_{22}]\rangle
\end{array}$& 
$\phi=
	\begin{pmatrix}
	x &    0  &  0 & 0\\
	y &  x^2  &  0 & 0\\
	z &  2xy  & x^3 &  0\\
	t &   y^{2}+2xz & 3 x^{2}y & x^4
	\end{pmatrix} 
	$\\
\hline
\end{longtable}

Let us use the following notations:

\begin{longtable}{ll}
$\nabla_1 = [\Delta_{14}]+[\Delta_{23}]$& $\nabla_2 = [\Delta_{22}]$
\end{longtable}

Take $\theta=\sum\limits_{i=1}^2\alpha_{i}\nabla_{i}\in {\rm H_{\mathfrak{CCD}}^2}(\mathfrak{C}^{4*}_{09}).$
We are interested only in $(\af_1,\af_2) \neq (0,0)$ and can suppose $\af_2=1.$
	Since
 	$$
	\phi^T\begin{pmatrix}
	0 & 0  & 0 & \alpha_1 \\
	0  & 1  & \alpha_1 & 0\\
	0 & \alpha_1 & 0  & 0\\
	\alpha_1 & 0 & 0   & 0
	\end{pmatrix} \phi=
	\begin{pmatrix}
	\alpha^{***}  & \alpha^{**}   & \alpha^{*} & \alpha_{1}^{*}\\
	\alpha^{**}  & \alpha_{2}^{*}+\alpha^{*}  & \alpha_{1}^{*} & 0 \\
	\alpha^{*} &  \alpha_{1}^{*}& 0 & 0\\
	\alpha_{1}^{*} & 0 & 0 & 0
	\end{pmatrix},
	$$ 
where
\begin{longtable}{lclccclcl}
$\alpha^*_1$&$ =$&$ x^5\alpha_{1}$& &&&$
\alpha^*_2 $&$= $&$x^4$
\end{longtable}
Here we have only one   representatives $\langle \nabla_1 + \nabla_2\rangle.$
 
Hence we have the following $5$-dimensional algebra:
\begin{longtable}{llllllllll}
$\mathfrak{C}_{78 }^5$& $ : $ & $e_1 e_1 = e_2$ & $e_1 e_2=e_3$ & $e_1e_3=e_4$ &$e_1e_4=e_5$&$e_2e_2=e_4+e_5$ &$e_2e_3=e_5$\\ 
\end{longtable}

 \subsection{$1$-dimensional central extensions of ${\mathfrak{C}_{02}^{4}}(\af)$}
Here we will  collect all information about ${\mathfrak{C}_{02}^{4}}(\af):$

\begin{longtable}{|l|l|l|}
\hline
Algebra  & Multiplication & $\mathfrak{CCD}$-Cohomology   \\
\hline
${\mathfrak{C}}_{02}^{4}(\af)$ &  
$\begin{array}{l}
e_1 e_1 = e_2\\ 
e_1 e_2=e_3 \\
e_1e_3= \alpha e_4\\ 
e_2e_2= (\alpha +1)e_4\end{array}$
& 
$\begin{array}{lcl}
{\rm H^2_{\mathfrak{CCD}}}({\mathfrak{C}}_{02}^{4}(\af)_{\af\neq0,1})&=&\langle [\Delta_{22}], [\Delta_{14}+(3+\alpha)  \Delta_{23}]\rangle\\

{\rm H^2_{\mathfrak{CCD}}}({\mathfrak{C}}_{02}^{4}(1))&=&\langle [\Delta_{22}], [\Delta_{14}+4\Delta_{23}],  [\Delta_{24}]\rangle\\

{\rm H^2_{\mathfrak{CCD}}}({\mathfrak{C}}_{02}^{4}(0))&=&\langle [\Delta_{13}], [\Delta_{14}+3\Delta_{23}] \rangle
 
 \end{array}$
 \\
 \hline
 \multicolumn{3}{|l|}{ Automorphisms}\\
\hline	
\multicolumn{3}{|c|}{  
$\phi=
	\begin{pmatrix}
	x &    0  &  0 & 0\\
	q &  x^2  &  0 & 0\\
	w &  2 q x  & x^4 & 0\\
	e & q^2 + \af  q^2 + 2 \af x w & (1 + 3 \af) x^2 y & x^3
	\end{pmatrix} 
	$}\\
\hline
\end{longtable}

\subsubsection{The case $\af\neq 0,1$}
Let us use the following notations:
\begin{longtable}{lcl}
$\nabla_1=[\Delta_{22}]$ & $\nabla_2=[\Delta_{14}+(3+\af)\Delta_{23}]$
\end{longtable}

Take $\theta=\sum\limits_{i=1}^2\alpha_{i}\nabla_{i}\in {\rm H_{\mathfrak{CCD}}^2}(\mathfrak{C}_{02}^4(\af)_{\af\neq 0,1}).$
We are interested only in $\af_2\neq 0$ and we can suppose $\af_2=1.$
	Since
	$$
	\phi^T\begin{pmatrix}
	0 & 0  & 0 & 1 \\
	0  &  \alpha_1  & 3+\af  & 0\\
	0 & 3+\af  & 0  & 0\\
	1 & 0 & 0   & 0
	\end{pmatrix} \phi=
	\begin{pmatrix}
	\alpha^{***}  & \alpha^{**}   & \af \alpha^{*} & \alpha_{2}^{*}\\
	\alpha^{**}  & \alpha_{1}^{*}+(\af+1) \alpha^{*}  & (3+\af)\alpha_{2}^{*} & 0 \\
	\alpha^{*} &  (3+\af)\alpha_{2}^{*}& 0 & 0\\
	\alpha_{2}^{*} & 0 & 0 & 0
	\end{pmatrix},
	$$ 
we have
\begin{longtable}{lcl}
$\alpha^*_1$ &$=$&$x^3 (4 q \af_2 ( \af-1)+x \af_1 \af)\af^{-1} $\\
$\alpha^*_2$ &$=$&$ x^5$.
\end{longtable}

By choosing 
$w=0,$
$q=\frac{x \af_1 \af}{
4 \af_2(1 -  \af)},$ $e=0$
and $x=1,$ we have the representative $\langle \nabla_2 \rangle.$

\subsubsection{The case $\af=1$}
Let us use the following notations:
\begin{longtable}{lll}
$\nabla_1=[\Delta_{22}]$ & $\nabla_2=[\Delta_{14}]+4 [\Delta_{23}]$ & $\nabla_3=[\Delta_{24}]$
\end{longtable}

Take $\theta=\sum\limits_{i=1}^3\alpha_{i}\nabla_{i}\in {\rm H_{\mathfrak{CCD}}^2}(\mathfrak{C}_{02}^4(1)).$
We are interested only in $(\af_2, \af_3)\neq (0,0).$
	Since
	$$
	\phi^T\begin{pmatrix}
	0 & 0  & 0 & \alpha_2 \\
	0  &  \alpha_1  & 4\alpha_2 & \af_3\\
	0 & 4\alpha_2 & 0  & 0\\
	\alpha_2 & \af_3 & 0   & 0
	\end{pmatrix} \phi=
	\begin{pmatrix}
	\alpha^{***}  & \alpha^{**}   &  \alpha^{*} & \alpha_{1}^{*}\\
	\alpha^{**}  & \alpha_{2}^{*}+2 \alpha^{*}  &4 \alpha_{1}^{*} & \af_3^* \\
	\alpha^{*} &  4 \alpha_{1}^{*}& 0 & 0\\
	\alpha_{1}^{*} & \af_3^* & 0 & 0
	\end{pmatrix},
	$$ 
where
\begin{longtable}{lcl}
$\alpha^*_1$ &$=$&$ x^2 (x^2 \af_1-4 q^2 \af_3+4 w x \af_3)$\\
$\alpha^*_2$ &$=$&$ x^4 (x \af_2+q \af_3)$\\
$\alpha^*_3$ &$=$&$ x^6 \af_3$\\
\end{longtable}
Then we have the following cases:
\begin{enumerate}

\item if $\af_3=0,$ then we have two representatives
$\langle \nabla_2 \rangle$ and $\langle \nabla_1+\nabla_2 \rangle,$ depending on $\af_1=0$ or not;

\item if $\af_3\neq 0,$ then choosing  $x=4,$
$q=-4 \af_2,$
$w=4 \af_2^2-\af_1$
and $e=0,$
we have the representative $\langle \nabla_3 \rangle.$

\end{enumerate}

\subsubsection{The case $\af=0$}
Let us use the following notations
\begin{longtable}{lclccclcl}
$\nabla_1 = [\Delta_{13}]$& $\nabla_2 = [\Delta_{14}]+3[\Delta_{23}]$ 
\end{longtable}
	Take $\theta=\sum\limits_{i=1}^2\alpha_{i}\nabla_{i}\in {\rm H_{\mathfrak{CCD}}^2}(\mathfrak{C}_{02}^{4}(0)).$ We are interested only in $\af_2\neq 0$ and can suppose $\af_2=1.$ Since
 	$$
	\phi^T\begin{pmatrix}
	0  &  0  &  \alpha_1  & \alpha_2\\
	0  &  0  &  3\alpha_2 &0 \\
	\af_1 & 3\alpha_2  & 0   & 0\\
	\alpha_2 & 0  & 0   & 0
	\end{pmatrix} \phi=
	\begin{pmatrix}
	\alpha^*      &  \alpha^{**}    & \alpha^*_1 & \alpha^{*}_2\\
	\alpha^{**}    & \alpha^{***}     &3\alpha^*_2& 0\\
	\af^*_1 & \alpha^*_2  &  0  & 0 &  \\
	\alpha^{*}_2 & 0 & 0 & 0
	\end{pmatrix},
	$$ 
	we have
	\begin{longtable}{rclcccrcl}
$\alpha^*_1 = x^3 (4 q \af_2+x \af_1)$
& $\alpha^*_2 = x^5\alpha_{1}$.	\end{longtable}

By choosing $q=-\frac{\alpha_1}{4\alpha_2},$  we have the representative $\langle \nabla_2 \rangle$.

Summarizing all cases from  the family of algebras  ${\mathfrak{C}_{02}^{4}}(\af),$
we have the following distinct orbits

\begin{longtable}{llllllllll}
$\langle \nabla_1+\nabla_2 \rangle_{\af=1} $ & 
$\langle \nabla_2 \rangle $ & 
$\langle \nabla_3 \rangle_{\af=1} $
\end{longtable}

Hence we have the following $5$-dimensional algebra:
\begin{longtable}{llllllllll}

$\mathfrak{C}_{79}^5 $& $ : $ & 
$e_1 e_1 = e_2$& 
$e_1 e_2=e_3$&
$e_1e_3=  e_4$& 
$e_1e_4=  e_5$& 
$e_2e_2= 2 e_4+e_5$&
$e_2e_3=  4 e_5$ \\

$\mathfrak{C}_{80}^5(\af)$& $ : $ & 
$e_1 e_1 = e_2$& 
$e_1 e_2=e_3$&
$e_1e_3= \alpha e_4$& 
$e_1 e_4=e_5$&
$e_2e_2= (\alpha +1)e_4$&
$e_2 e_3=(\af+3)e_5$\\

$\mathfrak{C}_{81}^5 $& $ : $ & 
$e_1 e_1 = e_2$& 
$e_1 e_2=e_3$&
$e_1e_3=  e_4$& 
$e_2e_2= 2 e_4$&
$e_2e_4=   e_5$ \\

\end{longtable}

\subsection{ Trivial central extensions}
Here we collect the cocycles  of other $4$-dimensional nilpotent $\mathfrak{CCD}$-algebras.
It is easy to see, that every central extension of these algebras is trivial.
 
\begin{longtable}{|l|l|}\hline

\hline

Algebras & Cocycles\\
\hline

${\mathfrak{C}_{07}^{4*}},$
${\mathfrak{C}_{11}^{4*}},$
${\mathfrak{C}_{12}^{4*}},$
${\mathfrak{C}^{4*}_{01}},$
${\mathfrak{C}}^{4}_{03},$
${\mathfrak{C}}^{4}_{04},$
${\mathfrak{C}}^{4}_{05},$
${\mathfrak{C}}^{4}_{06},$
${\mathfrak{C}}^{4}_{07}$ &
$\left\langle \Delta_{11}, \Delta_{12}, \Delta_{22}, \Delta_{13}, \Delta_{23}, \Delta_{33} \right\rangle$ \\

\hline

\end{longtable}

\medskip

  \section{Classification theorem}
  
  \begin{theorem}\label{teor}
Let $\mathfrak{C}$ be a complex $5$-dimensional nilpotent commutative $\mathfrak{CD}$-algebra.
Then $\mathfrak{C}$ is a Jordan algebra or it is isomorphic to one algebra from the following list:
{\tiny 
\begin{longtable}{lllllllllll}

$\mathfrak{C}^{5}_{01}$&$:$& $e_1 e_1 = e_2$ & $e_2 e_2=e_3$\\
			            
$\mathfrak{C}^{5}_{02}(\alpha)$&$:$&  $e_1 e_1 = e_2$  & $e_1 e_2=e_3$& $e_1e_3= \alpha e_4$  & \multicolumn{2}{l}{$e_2e_2= (\alpha +1)e_4$} \\
			
$\mathfrak{C}^{5}_{03}$& $: $&   $e_1 e_1 = e_2$& $e_1e_3=e_4$& $e_2e_2=e_4$		\\

$\mathfrak{C}^{5}_{04}$& $: $&     $e_1 e_1 = e_2$ & $e_2e_2=e_4$& $e_3e_3=e_4$ \\
 
$\mathfrak{C}^5_{05}$ & $: $ & $e_1 e_2 = e_3$ & $e_3 e_3=e_4$	\\  
$\mathfrak{C}^5_{06}$ & $: $ & $e_1 e_1 = e_4$ & $e_1 e_2=e_3$ & $e_2e_2=e_4$& $e_3e_3=e_4$\\ 
$\mathfrak{C}^5_{07}$ & $: $ & $e_1 e_1 = e_4$ & $e_1 e_2=e_3$ & $e_3e_3=e_4$\\
$\mathfrak{C}^5_{08}$ & $: $ & $e_1e_1=e_2$ & $e_1e_3=e_4$ & $e_2e_2=e_5 $\\
$\mathfrak{C}^5_{09}$ & $: $& $e_1e_1=e_2$ & $e_1e_3=e_4$ & $e_2e_2=e_5$ & $e_3e_3=e_5 $\\
$\mathfrak{C}^5_{10}$ & $ :$ & $e_1e_1=e_2$ & $e_1e_2=e_4$ & $e_2e_2=e_5 $\\
$\mathfrak{C}^5_{11}$ & $ : $ & $e_1e_1=e_2$ & $e_1e_2=e_4$ & $e_1e_3=e_5$ & $e_2e_2=e_5 $\\
$\mathfrak{C}_{12}^5(\alpha)$&$:$& 
$e_1e_1=e_2$ & $e_1e_2=e_3$ &$e_1e_3=(\alpha+1) e_5$ &$e_2e_2= \alpha e_5$  &$e_2e_4= e_5$\\
$\mathfrak{C}_{13}^5(\alpha, \beta)$&$:$& 
$e_1e_1=e_2$ & $e_1e_2=e_3$ &$e_1e_3=(\alpha+1) e_5$ &$e_2e_2= \alpha e_5$ &$e_2e_4= \beta e_5$ &$e_4e_4= e_5$\\
$\mathfrak{C}^5_{14}$ & $ : $ & $e_1e_1=e_2$ & $e_2e_2=e_5$ & $e_3e_3=e_4 $\\
$\mathfrak{C}^5_{15}$ & $ : $ & $e_1e_1=e_2$ & $e_1e_3=e_5 $ & $ e_2e_2=e_5$ & $e_3e_3=e_4 $\\
$\mathfrak{C}_{16}^5(\alpha)$& $ : $ & $e_1e_1=e_2$ & $e_1e_2=e_4$ &$e_1e_4= (\af+1) e_5$& $e_2e_2=\af e_5$ &$e_3e_3=e_4$\\
$\mathfrak{C}^5_{17}$ & $ : $ & $e_1e_1=e_2$ & $e_1e_2=e_4$ & $e_1e_3=e_5$ & $e_2e_2=e_5$ & $e_3e_3=e_4 $\\
$\mathfrak{C}^5_{18}$ & $ : $ & $e_1e_1=e_2$ & $e_2e_2=e_5$ & $e_2e_3=e_4 $\\
$\mathfrak{C}^5_{19}$ & $ : $ & $e_1e_1=e_2$ & $e_2e_2=e_5$ & $e_2e_3=e_4$ & $e_3e_3=e_5 $\\
$\mathfrak{C}^5_{20}$ & $ : $ & $e_1e_1=e_2$ & $e_1e_3=e_5$ & $e_2e_2=e_5$ & $e_2e_3=e_4 $\\
$\mathfrak{C}^5_{21}$ & $ : $ & $e_1e_1=e_2$ & $e_1e_3=e_5$ & $e_2e_2=e_5$ & $e_2e_3=e_4$ & $e_3e_3=e_5 $\\
$\mathfrak{C}^5_{22}$ & $ : $ & $e_1e_1=e_2$ & $e_1e_2=e_5$ & $e_2e_2=e_5$ & $e_2e_3=e_4 $\\
$\mathfrak{C}^5_{23}$ & $ : $ & $e_1e_1=e_2$ & $e_1e_2=e_5$ & $e_2e_2=e_5$ & $e_2e_3=e_4$ & $e_3e_3=e_5 $\\
$\mathfrak{C}^5_{24}(\alpha)$ & $ : $ & $e_1e_1=e_2$ & $e_1e_2=e_5$ & $e_1e_3=e_5$ & $e_2e_2=e_5$ & $e_2e_3=e_4$ & $e_3e_3=\alpha e_5 $\\
$\mathfrak{C}_{25}^5$&$:$& $e_1e_1=e_2$ & $e_1e_2=e_3$ &$e_1e_3=e_4$ & $e_2e_2=e_5$\\
$\mathfrak{C}^5_{26}(\alpha,\beta)$ & $: $ & $e_1e_1=\alpha e_5$ & $e_1e_2=e_3$ & $e_2e_2=\beta e_5$ & $e_1e_3=e_4+e_5$ & $e_2e_3=e_4$ & $e_3e_3=e_5 $\\
$\mathfrak{C}^5_{27}(\alpha)$ & $ : $ & $e_1e_1=\alpha e_5$ & $e_1e_2=e_3$ & $e_2e_2=e_5$ & $e_1e_3=e_4$ & $e_2e_3=e_4$ & $e_3e_3=e_5 $\\
$\mathfrak{C}^5_{28}$ & $ : $ & $e_1e_1=e_5$ & $e_1e_2=e_3$ & $e_1e_3=e_4$ & $e_2e_3=e_4$ & $e_3e_3=e_5 $\\
$\mathfrak{C}^5_{29}$ & $ : $ & $e_1e_2=e_3$ & $e_1e_3=e_4$ & $e_2e_3=e_4$ & $e_3e_3=e_5 $\\
$\mathfrak{C}^5_{30}(\alpha)$ & $ : $ & $e_1e_1=e_4+\alpha e_5$ & $e_1e_2=e_3$ & $e_2e_2=e_5$ & $e_2e_3=e_4$ & $e_3e_3=e_5 $\\
$\mathfrak{C}^5_{31}$ & $ : $ & $e_1e_1=e_4+e_5$ & $e_1e_2=e_3$ & $e_2e_3=e_4$ & $e_3e_3=e_5 $\\
$\mathfrak{C}^5_{32}$ & $ : $ & $e_1e_1=e_4$ & $e_1e_2=e_3$ & $e_2e_3=e_4$ & $e_3e_3=e_5 $\\
$\mathfrak{C}^5_{33}$ & $ : $ & $e_1e_1=e_5$ & $e_1e_2=e_3$ & $e_2e_2=e_5$ & $e_2e_3=e_4$ & $e_3e_3=e_5 $\\
$\mathfrak{C}^5_{34}$ & $ : $ & $e_1e_1=e_5$ & $e_1e_2=e_3$ & $e_2e_3=e_4$ & $e_3e_3=e_5 $\\
$\mathfrak{C}^5_{35}$ & $ :$ & $e_1e_2=e_3$ & $e_2e_2=e_5$ & $e_2e_3=e_4$ & $e_3e_3=e_5 $\\
$\mathfrak{C}^5_{36}$ & $ :$ & $e_1e_2=e_3$ & $e_2e_3=e_4$ & $e_3e_3=e_5 $\\
$\mathfrak{C}^5_{37}$ & $ :$ & $e_1e_1=e_4+e_5 $ & $e_1e_2=e_3$ & $e_2e_2=e_4$ & $e_3e_3=e_5 $\\
$\mathfrak{C}^5_{38}$ & $ :$ & $e_1e_1=e_4$ & $e_1e_2=e_3$ & $e_2e_2=e_4$ & $e_3e_3=e_5 $\\
$\mathfrak{C}^5_{39}$ & $ :$ & $e_1e_1=e_5$ & $e_1e_2=e_3$ & $e_2e_2=e_4$ & $e_3e_3=e_5 $\\
$\mathfrak{C}^5_{40}$ & $ : $ & $e_1e_2=e_3$ & $e_2e_2=e_4$ & $e_3e_3=e_5 $ \\
$\mathfrak{C}_{41}^5$& $ : $ & $e_1e_1=e_2$ & $e_2e_2=e_5$ &$e_3e_4=e_5$\\
$\mathfrak{C}_{42}^5$& $ : $ & $e_1e_1=e_2$ & $e_1e_3=e_5$ &$e_2e_2=e_5$ &$e_4e_4= e_5$\\
$\mathfrak{C}_{43}^5$& $ : $ & $e_1e_2=e_3$ & $e_1e_1=e_5$ & $e_2e_4=e_5$ & $e_3e_3=e_5$\\
$\mathfrak{C}_{44}^5$& $ : $ & $e_1e_2=e_3$ & $e_1e_1=e_5$ & $e_2e_2=e_5$ & $e_3e_3=e_5$& $e_4e_4=e_5$\\
$\mathfrak{C}_{45}^5$& $ : $ & $e_1e_2=e_3$ & $e_1e_1=e_5$ & $e_3e_3=e_5$ & $e_4e_4=e_5$\\
$\mathfrak{C}_{46}^5$& $ : $ & $e_1e_2=e_3$ & $e_1e_4=e_5$ & $e_2e_4=e_5$ & $e_3e_3=e_5$\\
$\mathfrak{C}_{47}^5$& $ : $ & $e_1e_2=e_3$ & $e_2e_4=e_5$ & $e_3e_3=e_5$ \\
$\mathfrak{C}_{48}^5$& $ : $ & $e_1e_2=e_3$ & $e_3e_3=e_5$ & $e_4e_4=e_5$\\
 $\mathfrak{C}_{49}^5(\af)$& $ : $ & $e_1e_1=e_3$ & $e_1e_2=e_5$ &$e_2e_2=e_4$ & $e_3e_3=\af e_5$ & $e_3e_4=e_5$& $e_4e_4=e_5$ \\
 $\mathfrak{C}_{50}^5$& $ : $ & $e_1e_1=e_3$  & $e_1e_2=e_5$ &$e_2e_2=e_4$ & $e_3e_3=e_5$ & $e_4e_4=e_5$\\
 $\mathfrak{C}_{51}^5$& $ : $ & $e_1e_1=e_3$  & $e_1e_2=e_5$ &$e_2e_2=e_4$ & $e_3e_4=e_5$\\
 $\mathfrak{C}_{52}^5(\af)$& $ : $ & $e_1e_1=e_3$  & $e_1e_3=\af e_5$ &$e_2e_2=e_4$ &$e_2e_3=e_5$ &$e_3e_3=e_5$
 &$e_3e_4=e_5$ &$e_4e_4=e_5$\\
 $\mathfrak{C}_{53}^5$& $ : $ & $e_1e_1=e_3$  & $e_1e_3=e_5$ &$e_2e_2=e_4$ & $e_2e_3=e_5$ & $e_4e_4=e_5$\\
 $\mathfrak{C}_{54}^5$& $ : $ & $e_1e_1=e_3$  & $e_1e_3=e_5$ &$e_2e_2=e_4$ & $e_4e_4=e_5$\\
 $\mathfrak{C}_{55}^5$& $ : $ & $e_1e_1=e_3$  & $e_2e_2=e_4$ &$e_2e_3=e_5$ & $e_4e_4=e_5$\\
 $\mathfrak{C}_{56}^5(\af)$& $ : $ & $e_1e_1=e_3$  &$e_2e_2=e_4$ & $e_3e_3=\alpha e_5$ & $e_3e_4=e_5$ & $e_4e_4=e_5$\\
 $\mathfrak{C}_{57}^5$& $ : $ & $e_1e_1=e_3$  &$e_2e_2=e_4$ & $e_3e_3=e_5$ & $e_4e_4=e_5$\\
 $\mathfrak{C}_{58}^5$& $ : $ & $e_1e_1=e_3$  &$e_2e_2=e_4$ & $e_3e_4=e_5$\\
 $\mathfrak{C}_{59}^5$& $ : $ & $e_1e_1=e_3$ & $e_1e_2=e_4$ &$e_1e_3=e_5$ &$e_2e_2=e_5$ &$e_4e_4=e_5$\\
$\mathfrak{C}_{60}^5$& $ : $ & $e_1e_1=e_3$ & $e_1e_2=e_4$ &$e_1e_3=e_5$ &$e_2e_3=e_5$ &$e_4e_4=e_5$\\
$\mathfrak{C}_{61}^5$& $ : $ & $e_1e_1=e_3$ & $e_1e_2=e_4$ &$e_1e_3=e_5$ &$e_4e_4=e_5$\\ 
$\mathfrak{C}_{62}^5(\af)$& $ : $ & $e_1e_1=e_3$ & $e_1e_2=e_4$&$e_1e_4=e_5$ &$e_2e_2=\af e_5$ &$e_3e_3=e_5$\\ 
$\mathfrak{C}_{63}^5$& $ : $ & $e_1e_1=e_3$ & $e_1e_2=e_4$ &$e_2e_2=e_5$ &$e_3e_4=e_5$\\
$\mathfrak{C}_{64}^5$& $ : $ & $e_1e_1=e_3$ & $e_1e_2=e_4$ &$e_2e_3=e_5$ &$e_3e_3=e_5$ &$e_4e_4=e_5$\\ 
$\mathfrak{C}_{65}^5$& $ : $ & $e_1e_1=e_3$ & $e_1e_2=e_4$ &$e_2e_3=e_5$ &$e_4e_4=e_5$\\
$\mathfrak{C}_{66}^5$& $ : $ & $e_1e_1=e_3$ & $e_1e_2=e_4$ &$e_2e_4=e_5$ &$e_3e_3=e_5$\\ 
$\mathfrak{C}_{67}^5$& $ : $ & $e_1e_1=e_3$ & $e_1e_2=e_4$ &$e_3e_3=e_5$ &$e_4e_4=e_5$\\ 
$\mathfrak{C}_{68}^5$& $ : $ &   $e_1e_1=e_3$ & $e_1e_2=e_4$ &$e_3e_4=e_5$\\
${\mathfrak{C}}_{69}^{5}(\af)$ & $ : $ &   $e_1 e_1 = e_4$&$e_1e_2=\af e_5$ &$e_1e_3=e_5$ &$e_2e_2=e_5$
&$ e_2 e_3=e_4$ &$e_4e_4=e_5$\\
${\mathfrak{C}}_{70}^{5}$ & $ : $ &   $e_1 e_1 = e_4$&$e_1e_2=e_5$ &$e_1e_3=e_5$ &$ e_2 e_3=e_4$ &$e_4e_4=e_5$ \\
${\mathfrak{C}}_{71}^{5}$& $ : $ &   $e_1 e_1 = e_4$&$e_1e_2=e_5$&$ e_2 e_3=e_4$ &$e_4e_4=e_5$ \\
${\mathfrak{C}}_{72}^{5}$ & $ : $ &   $e_1 e_1 = e_4$&$e_2e_2=e_5$&$ e_2 e_3=e_4+e_5$ &$e_4e_4=e_5$ \\
${\mathfrak{C}}_{73}^{5}$ & $ : $ &   $e_1 e_1 = e_4$&$e_2e_2=e_5$&$ e_2 e_3=e_4$ &$e_4e_4=e_5$ \\
${\mathfrak{C}}_{74}^{5}$ & $ : $ &   $e_1 e_1 = e_4$&$ e_2 e_3=e_4+e_5$ &$e_4e_4=e_5$ \\
${\mathfrak{C}}_{75}^{5}$ & $ : $ &   $e_1 e_1 = e_4$&$ e_2 e_3=e_4$ &$e_4e_4=e_5$ \\
$\mathfrak{C}_{76}^5$& $ : $ & $e_1e_1=e_2$ & $e_1e_2=e_4$  &$e_1e_4= e_5$ &$e_2e_2= - 2 e_5$ &$e_3e_3= e_4+3e_5$\\
$\mathfrak{C}_{77}^5$& $ : $ & $e_1e_1=e_2$ & $e_1e_2=e_4$ &$e_1e_4= e_5$ &$e_2e_3=  e_5$ &$e_3e_3=e_4$\\
 
$\mathfrak{C}_{78 }^5$& $ : $ & $e_1 e_1 = e_2$ & $e_1 e_2=e_3$ & $e_1e_3=e_4$ &$e_1e_4=e_5$&$e_2e_2=e_4+e_5$ &$e_2e_3=e_5$\\ 

$\mathfrak{C}_{79}^5 $& $ : $ & 
$e_1 e_1 = e_2$& 
$e_1 e_2=e_3$&
$e_1e_3=  e_4$& 
$e_1e_4=  e_5$& 
$e_2e_2= 2 e_4+e_5$&
$e_2e_3=  4 e_5$ \\

$\mathfrak{C}_{80}^5(\af)$& $ : $ & 
$e_1 e_1 = e_2$& 
$e_1 e_2=e_3$&
$e_1e_3= \alpha e_4$& 
$e_1 e_4=e_5$&
$e_2e_2= (\alpha +1)e_4$&
\multicolumn{2}{l}{$e_2 e_3=(\af+3)e_5$}\\

$\mathfrak{C}_{81}^5 $& $ : $ & 
$e_1 e_1 = e_2$& 
$e_1 e_2=e_3$&
$e_1e_3=  e_4$& 
$e_2e_2= 2 e_4$&
$e_2e_4=   e_5$ \\
		\end{longtable}}

All algebras from the present list are non-isomorphic, excepting		
\begin{longtable}{lcccccr}
$\mathfrak{C}_{13}^5(\alpha, \beta) \cong         
 \mathfrak{C}_{13}^5(\alpha,-\beta)$ & \ & $\mathfrak{C}^5_{26}(\alpha,\beta) \cong  \mathfrak{C}^5_{26}(\beta, \alpha)$ & \ & $\mathfrak{C}^5_{27}(\alpha) \cong  \mathfrak{C}^5_{27}(\frac 1 {\alpha})$ & \ &
${\mathfrak{C}}_{69}^{5}(\af) \cong  {\mathfrak{C}}_{69}^{5}(\sqrt[3]{1}\af)$ 
\end{longtable}
  \end{theorem}

\end{document}